\documentclass[a4paper]{amsart} 

\usepackage[colorlinks,linkcolor=blue,citecolor=blue,urlcolor=blue]{hyperref}
\usepackage[english]{babel}
\usepackage{amsthm,amsfonts,amssymb,mathrsfs,enumerate,amsmath}
\usepackage{xcolor}
\usepackage[small,nohug,heads=vee]{diagrams}
\diagramstyle[labelstyle=\scriptstyle]

\usepackage{mathtools}						
\mathtoolsset{showonlyrefs,showmanualtags}	

\usepackage{todonotes}  

\newtheorem{Theorem}{Theorem}[section]
\newtheorem{Coro}[Theorem]{Corollary}
\newtheorem{Lemma}[Theorem]{Lemma}
\newtheorem{Prop}[Theorem]{Proposition}

\theoremstyle{definition}
\newtheorem{Def}[Theorem]{Definition}

\theoremstyle{remark}
\newtheorem{rk}[Theorem]{Remark}

\renewcommand{\aa}{\alpha}

\newcommand{\eps}{\varepsilon}

\newcommand{\R}{\mathbb{R}}

\newcommand{\N}{\mathbb{N}}

\numberwithin{equation}{section}

\title[On projective and affine equivalence of sub-Riemannian metrics]{On projective and affine equivalence of sub-Riemannian metrics}
\date{\today}

\thanks{This work was supported by a public grant as part of the
Investissement d'avenir project, reference ANR-11-LABX-0056-LMH, LabEx
LMH, in a joint call with Programme Gaspard Monge en Optimisation et
Recherche Op\'{e}rationnelle, by the iCODE
Institute project funded by the IDEX Paris-Saclay, ANR-11-IDEX-0003-02 and by the Grant ANR-15-CE40-0018 of the ANR. I.\ Zelenko is partly supported by NSF grant DMS-1406193 and Simons Foundation Collaboration Grant for Mathematicians 524213.}

\author{Fr\'ed\'eric Jean}
\address{Fr\'ed\'eric Jean \\
Unit\'{e} de Math\'{e}matiques Appliqu\'{e}es, ENSTA ParisTech, Universit\'{e} Paris-Saclay,
91120 Palaiseau, France}
\email{frederic.jean@ensta-paristech.fr}
\urladdr{\url{http://uma.ensta-paristech.fr/~fjean}}

\author{Sofya Maslovskaya}
\address{Sofya Maslovskaya \\
Unit\'{e} de Math\'{e}matiques Appliqu\'{e}es, ENSTA ParisTech, Universit\'{e} Paris-Saclay,
91120 Palaiseau, France}
\email{sofya.maslovskaya@ensta-paristech.fr}

\author{Igor Zelenko}
\address{Igor Zelenko\\
         Department of Mathematics\\
         Texas A\&M University\\
         College Station\\
         Texas \ 77843\\
         USA}
\email{zelenko@math.tamu.edu}
\urladdr{\url{http://www.math.tamu.edu/~zelenko}}

\begin{document}

\begin{abstract}
Consider a smooth connected manifold $M$ equipped with a bracket generating distribution $D$. Two sub-Riemannian metrics on $(M,D)$ are said to be \emph{projectively} (resp.\ \emph{affinely}) \emph{equivalent} if they have the same geodesics up to reparameterization (resp.\ up to affine reparameterization).
A sub-Riemannian metric $g$ is called \emph{rigid} (resp.\ \emph{conformally rigid}) with respect to projective/affine equivalence, if any sub-Riemannian metric  which is projectively/affinely equivalent to $g$ is constantly proportional to $g$ (resp.\ conformal to $g$).
In the Riemannian case the local classification of projectively (resp.\ affinely) equivalent metrics was done in the classical work \cite{Levi-Civita1896} (resp.\ \cite{Eisenhart1923}).
In particular, a Riemannian metric which  is not rigid with respect to one of the above equivalences satisfies the following two special properties: its geodesic flow  possesses a collection of  nontrivial integrals of special type and the metric induces certain canonical product structure on the ambient manifold.
The only proper sub-Riemannian cases to which these classification results were extended so far  are sub-Riemannian metrics on contact and quasi-contact distributions \cite{Z}.
The general goal is to extend these results to arbitrary sub-Riemannian manifolds. In this article we establish two types of results toward this goal:
if a sub-Riemannian metric is not conformally rigid with respect to the projective equivalence, then, first, its flow of normal extremals has at least one nontrivial integral quadratic on the fibers of the cotangent bundle and, second, the nilpotent approximation of the underlying distribution  at any point admits a product structure. As a consequence we obtain  two types of genericity results for rigidity:
first, we show that a generic sub-Riemannian metric on a fixed pair $(M,D)$ is conformally rigid with respect to projective equivalence. Second, we prove that, except for special pairs $(m,n)$, for a generic distribution $D$ of rank $m$ on an $n$-dimensional manifold, every sub-Riemannian metric on $D$ is conformally rigid with respect to the projective equivalence. For the affine equivalence in both genericity results conformal rigidity can be replaced by usual rigidity.
\end{abstract}

\maketitle

\tableofcontents

\section{Introduction}
In Riemannian geometry, projectively (or geodesically) equivalent metrics are Riemannian metrics on the same manifold which have the same geodesics, up to reparameterization. The study of equivalent metrics dates back to the works of Dini and Levi-Civita in the 19th century. The interest in this notion of equivalence is renewed by recent applications of optimal control theory to the study of human motor control. Indeed, finding the optimality criterion followed by a particular human movement amounts to solve what is called an \emph{inverse optimal control problem} (see for instance \cite{berret2008,chittaro2013}): given a set $\Gamma$ of trajectories and a class of optimal control problems -- that is, a pair (dynamical constraint, class $\mathcal{L}$ of infinitesimal costs) --, identify a cost function $L$ in
$\mathcal{L}$ such that the elements of $\Gamma$ are minimizing trajectories of the optimal control problem associated with the integral cost $\int L$. Being the solutions of a same inverse problem defines an equivalence between costs in $\mathcal{L}$ similar to projective equivalence for Riemannian metrics.
Our purpose here is to extend and study this kind of equivalence in the context of sub-Riemannian geometry. This is a first step in the direction of a more general goal, which is to give a rigorous theoretical framework to the study of inverse optimal control problems.

A sub-Riemannian manifold is a triple $(M,D,g)$, where $M$ is a smooth manifold, $D$ is a distribution on $M$ (i.e.\ a subbundle of $TM$) which is assumed to be bracket generating, and $g$ is a Riemannian metric on $D$. We say that $g$ is a sub-Riemannian metric on $(M,D)$. Riemannian geometry appears as the particular case where $D=TM$. A horizontal curve is an absolutely continuous curve tangent to $D$, and for such a curve $\gamma$  the length and the energy are defined as in Riemannian geometry by respectively $\ell(\gamma) = \int \sqrt{g(\dot{\gamma},\dot{\gamma})}$ and $E(\gamma) = \int g(\dot{\gamma},\dot{\gamma})$. A length minimizer (resp.\ an energy minimizer) is a horizontal curve minimizing the length (resp.\ the energy) among all the horizontal curves with the same extremities.

The length being independent on the parameterization of the curve, any time reparameterization of a length minimizer is still a length minimizer. On the other hand, a classical consequence of the Cauchy--Schwarz inequality is that the energy minimizers are the length minimizers with constant velocity, i.e.\ such that $g(\dot{\gamma},\dot{\gamma})$ is constant along $\gamma$.  It is then sufficient to describe the energy minimizers, the length minimizers being any time reparameterization of the latter.

It results from the Pontryagin Maximum Principle that energy minimizers are projections of Pontryagin extremals, and can be of two types, normal or abnormal geodesics. These geodesics play a role similar to the one of the solutions of the geodesic equation in Riemannian geometry. We thus extend the definition of equivalence of metrics in the following way.

\begin{Def}
Let $M$ be a manifold and $D$ be a bracket generating distribution on $M$. Two sub-Riemannian metrics  $g_1$ and $g_2$ on  $(M,D)$ are called \emph{projectively equivalent} at $q_0\in M$ if they have the same geodesics, up to a reparameterization, in a neighborhood of $q_0$. They are called \emph{affinely equivalent} at $q_0$ if they have the same geodesics, up to affine reparameterization, in a neighborhood of $q_0$.
\end{Def}

In the sequel the manifold $M$ is assumed to be connected.
The trivial example of equivalent metrics is the one of two constantly proportional metrics $g$ and $cg$, where $c>0$ is a real number. We thus say that these metrics are \emph{trivially} (projectively or affinely) equivalent.
Besides, affine equivalence implies projective equivalence but in general the two notions do not coincide. For instance, on $M=\R$, all metrics are projectively equivalent to each other while two metrics are affinely equivalent if and only if they are trivially equivalent.

Note that if two sub-Riemannian metrics on $(M,D)$ have the same set of length minimizers, then they are projectively equivalent. And if they have the same set of energy minimizers, then they are affinely equivalent. This results from the fact that on one hand normal geodesics are locally energy minimizers, and on the other hand abnormal geodesics are characterized only by the distribution $D$. Thus projective and affine equivalence are appropriate notions to study inverse optimal control problems where the dynamical constraint is $\dot \gamma \in D$ and the class $\mathcal{L}$ is the set of sub-Riemannian metrics. In particular, they allow one to answer to the following questions: given $M$ and $D$, can we recover $g$ in a unique way, up to a multiplicative constant, from the knowledge of all energy minimizers of $(M,D,g)$? And from the knowledge of all length minimizers of $(M,D,g)$? The answer to these questions is positive when the metric presents a kind of rigidity.

\begin{Def}
A metric $g$ on $(M,D)$ is said to be \emph{projectively rigid} (resp.\ \emph{affinely rigid}) if it admits no non-trivially projectively (resp.\ affinely) equivalent metric.
\end{Def}

We also introduce a weaker notion of rigidity associated with the concept of conformal metrics. Remind that a metric $g_2$ on $(M,D)$ is said to be \emph{conformal} to another metric $g_1$ if $g_2 = \alpha^2 g_1$ for some nonvanishing smooth function $\alpha: M \to \R$. The trivial case of constantly proportional metrics is the particular case where $\alpha$ is constant. Note that two conformal metrics are not projectively  equivalent in general (we actually conjecture that the latter situation occurs only when either $\dim M=1$ or the metrics are constantly proportional to each other).

\begin{Def}
A metric $g$ is said to be \emph{conformally projectively rigid} if any metric projectively equivalent to $g$ is conformal to $g$.
\end{Def}

It is easy to construct examples of metrics which are not projectively rigid. For example,  an Euclidean metric on a plane provides such an example. Indeed, its geodesics consist of straight lines. Take the Riemannian metrics on the same plane obtained by the pull-back from the round  metric on a sphere placed on this plane via the (inverse of) gnomonic projection, i.e. the stereographic projection with the center in the center of the sphere. Obviously the geodesic of this  metric are straight lines as  unparameterized curves geodesics but this metric is not constantly proportional to the original metric, because it has  nonzero constant Gaussian curvature (see \cite[Sect. 3.1]{jmz2017}). Note also that by a classical theorem by Beltrami \cite{Beltrami1869}, the metrics with constant sectional curvature are the only ones projectively equivalent to the flat ones.

If one extends the notion of equivalence to Lagrangians, then one arrives to the variational version of Hilbert's fourth problem in dimension $2$, which was solved by Hamel \cite{Hamel1903} and provides a very rich class of Lagrangians having straight lines as extremals.

Affine and projective equivalence of Riemannian metrics are actually both classical. From the results of Dini \cite{Dini1870}, it follows that under natural regularity assumptions a two-dimensional Riemannian metric is non projectively rigid if and only if it is a Liouville surface, i.e.,
its geodesic flow admits a non-trivial integral which is quadratic with respect to the velocities. This implies that generic Riemannian metrics on surfaces are projectively rigid. In \cite{Levi-Civita1896}, again under natural regularity assumptions, Levi-Civita proved that the same result holds for Riemannian metrics on manifolds of arbitrary dimensions and provided a classification of locally projectively equivalent Riemannian metrics. The affinely equivalent Riemannian metrics are exactly the metrics with the same Levi-Civita connection and the description of the pairs of Riemannian metrics with this property can be attributed to Eisenhart \cite{Eisenhart1923}. This description is also closely related to the de Rham decomposition of a Riemannian manifold and  the properties of its holonomy group \cite{deRham1952}.

These classical results in Riemannian case implies in particular that a Riemannian metric that is not rigid with respect to one of the above equivalences satisfies the following two special properties.

\begin{enumerate}
\item {\bf Integrability property.} Its geodesic flow  possesses a collection of nontrivial integrals quadratic on the fiber  and  in involution.
\item {\bf Product structure (or separation of variables) property.} Locally the ambient manifold $M$ is a product of at least two manifolds such that the metric is a product of metrics on the factors in the affine case and a sort of twisted product of Riemannian metrics on the factors in the projective case (for a precise meaning of twisting here see formula \eqref{met1} below).
\end{enumerate}

Note that similar relations between separability of the Hamilton-Jacobi equation on a Riemannian manifold and integrability (existence of Killing tensors) were extensively studied by Benenti \cite{benenti1, benenti2}, while a more conceptual explanation of the integrability property, based on the modern language of symplectic geometry was given by Matveev and Topalov \cite{Matveev-Topalov2003}.

In a proper sub-Riemannian case, the only complete classification of projectively equivalent metrics was done
far more recently by Zelenko \cite{Z} for contact and quasi-contact sub-Riemannian metrics.
The general goal is to extend the above classification results to an arbitrary  sub-Riemannian case. By analogy with the Levi-Civita classification we define a wide class of pairs of sub-Riemannian metrics that are projective equivalent, see Section \ref{se:LCpairs} and Appendix \ref{se:proof_LCpairs}. We call them the \emph{(generalized) Levi-Civita pairs}. Note that the generalized Levi-Civita pairs  satisfy both integrability and product structure properties. It turns out that the result of \cite{Z} about the contact and quasi-contact case can be actually {reformulated} in the following way: under a natural regularity assumption the generalized Levi-Civita pairs are the only pairs of projectively equivalent metrics. The natural question is whether  this is the case for arbitrary sub-Riemannian case, i.e.\ \emph{whether under some natural regularity assumption the generalized Levi-Civita pairs are the only pairs of projectively equivalent metrics.}

In the present paper we make several steps toward answering this question by proving  the following two general results, which are weaker than the integrability and product structure properties formulated above, but support them. The first result is the existence of at least one integral, which supports the integrability property.

\begin{Theorem}
\label{th:integral}
If a sub-Riemannian metric $g$ is not conformally projectively rigid, then its flow of normal extremals has at least one nontrivial (i.e.\ not equal to the sub-Riemannian Hamiltonian) integral quadratic on the fibers.
\end{Theorem}

The second result states that the product structure properties hold at the level of the nilpotent approximations.

\begin{Theorem}
\label{th:nilpotent_equiv_main}
Let $M$ be a smooth manifold, $D$ be a distribution on $M$, and $g_1, g_2$ be two sub-Riemannian metrics on $(M, D)$.
If $g_1, g_2$ are projectively equivalent and non conformal to  each other, then for $q$ in an open and dense subset of $M$, the nilpotent approximation $\hat{D}$ of $D$ at $q$ admits a product structure\footnote{The notion of ``Product structure'' has a specific meaning here, see Definition~\ref{def:product}}, and the nilpotent approximations $\hat{g}_1, \hat{g}_2$ of the metrics form a Levi-Civita pair with constant coefficients.
\end{Theorem}
Finally we prove in Corollary~\ref{le:conf_affine} that any conformally projectively rigid metric is affinely rigid. As a direct consequence of this and of Theorem \ref{th:nilpotent_equiv_main} we obtain the following.

\begin{Coro}
	\label{Cor:anyrank}
	Let $D$ be a bracket generating distribution on a connected manifold $M$. If at every point of an open and dense subset of $M$ the nilpotent approximation of $(M,D)$ does not admit a product structure, then any sub-Riemannian metric on $D$ is conformally projectively rigid and affinely rigid.
\end{Coro}

In particular, since no bracket generating rank-$2$ distribution admits a product structure, we have the following consequence of Theorem~\ref{th:nilpotent_equiv_main}.

\begin{Coro}
\label{Cor:rank2}
Any {bracket generating}  sub-Riemannian metric on a rank-2 distribution is conformally projectively rigid and affinely rigid on each connected complement of the underlying manifolds $M$.
\end{Coro}

Theorem \ref{th:integral} allows us to get the following rigidity property of generic sub-Riemannian metrics on a given distribution.

\begin{Theorem}
\label{th:generic_metric}
Let $M$ be a smooth connected manifold and $D$ be a distribution on $M$. A generic sub-Riemannian metric on $(M,D)$ is affinely rigid and conformally projectively rigid.
\end{Theorem}

Theorem \ref{th:nilpotent_equiv_main} allows us to get the following rigidity results for all sub-Riemannian metric of a generic distributions.

\begin{Theorem}
\label{th:generic_distrib}
Let $m$ and $n$ be two integers such that $2\leq m <n$, and assume $(m,n) \neq (4,6)$ and $m \neq n-1$ if $n$ is even.
Then, given an $n$-dimensional connected manifold $M$ and  a generic rank-$m$ distribution $D$ on $M$,
any sub-Riemannian metric on $(M,D)$ is affinely rigid and conformally projectively rigid.
\end{Theorem}

Few words now about the main ideas behind the proofs with references to the corresponding sections of the paper. The problem of the projective equivalence of sub-Riemannian metrics can be reduced to the problem of existence of a fiber preserving orbital diffeomorphism between the flows of normal extremals  in the cotangent bundle (orbital diffeomorphism means that it sends normal extremals of one metric to the normal extremals of another one considered as unparameterized curves). In the Riemannian case, if such diffeomorphism exists then it can be easily expressed in terms of the metrics.
It is not the case anymore in the proper sub-Riemannian case, which is the main difficulty here.
The reason is that, in contrast to the Riemannian case, a sub-Riemannian geodesic is not uniquely determined by its initial point and the velocity at this point (i.e.\ by its first jet at one point).
The order of jet which is needed to determine a geodesic uniquely is controlled by the flag of the Jacobi curves along the corresponding extremal, which were introduced in \cite{li1, li2}. In subsections \ref{jacsec} and \ref{se:ample}  we collect all necessary information about Jacobi curves in order to justify the reduction of the equivalence problem to the existence of a fiberwise diffeomorphism in subsection \ref{se:orbital}.

In what follows, for shortness a function which is a polynomial or rational function on each fiber of $T^*M$ will be simply called a polynomial or rational function respectively on $T^*M$.
The equations on a fiber preserving orbital diffeomorphism form a highly overdetermined system of differential equations. In subsection \ref{se:coord_orbdiffeo}  after certain prolongation process, we reduce this system  to a system of infinitely many linear algebraic equations with coefficients which are polynomial functions so that if a solution of this system exists, then it is unique. We refer to this system as the \emph{fundamental algebraic system} for orbital diffeomorphism.  Its solution must be a rational function involving quadratic radicals on $T^*M$.

 The analysis of compatibility conditions for this system leads to a set of algebraic conditions. In particular, we show
 that one specific polynomial function on $T^*M$ is divisible by another specific polynomial function on $T^*M$. This divisibility condition is equivalent to the existence of an integral 
 for the normal extremal flow, which proves Theorem \ref{th:integral}. In Appendix \ref{se:proof_quadratic} we prove that the non-existence of a non-trivial integral for the geodesic flow of a sub-Riemannian metric is a generic property, adapting the proof of the analogous result for the Riemannian case from \cite{Kruglikov-Matveev2015}. This implies Theorem \ref{th:generic_metric}.

The idea of the proof of Theorem \ref{th:nilpotent_equiv_main} comes from the fact that the filtration of the tangent bundle, generated by the iterative brackets of vector fields tangent to the underlying distribution, induces weighted degrees for polynomial function on $T^*M$. If we replace all coefficients of the fundamental algebraic system at a point by the components of the highest weighted degree, we will get exactly the fundamental algebraic system for the orbital diffeomorphism related to the flow of normal extremals of the nilpotent approximation of the first metric, see the proof of Theorem \ref{th:nilpotent_equiv}. This and the analysis of conditions for projective equivalence for left invariant sub-Riemannian metrics given in Theorem \ref{th:Carnot.prod} are the main steps of the proof of Theorem \ref{th:nilpotent_equiv_main}.

Finally, to prove Theorem \ref{th:generic_distrib} we analyse in Section \ref{sec:gen} for which pairs $(m,n)$ the generic $n$-dimensional graded  nilpotent Lie algebras generated by the homogeneous component of weight $1$ can not be represented as a direct sum of  two nonzero graded  nilpotent Lie algebras.

\section{Preliminaries}
\label{se:prelim}
\subsection{Sub-Riemannian manifolds}
Let us recall some standard notions from sub-Riemannian geometry. Let $M$ be a $n$-dimensional smooth manifold and $D$ be a rank-$m$ distribution on $M$, i.e., $D$ is  a subbundle of $TM$ of rank $m$. We define by induction a sequence of modules of vector fields by setting $D^1 = \{ X : X \text{ is a section of } D \}$, and, for any integer $k>1$,  $D^k = D^{k-1} + \mathrm{span}\{ [X, Y] : X \text{ is a section of } D, \ Y \text{ belongs to } D^{k-1}\}$, where span is taken over the smooth functions on $M$.
	The \emph{Lie algebra} $\mathrm{Lie}(D)$ generated by the distribution $D$ is defined as  $\mathrm{Lie}(D) = \bigcup_{k \geq 1} D^k$.	

For $q \in M$, we denote by $D^k(q)$ and $\mathrm{Lie}(D)(q)$ respectively  the subspaces $\mathrm{span}\{X(q) : X \in D^k \}$ and $\mathrm{span}\{X(q) : X \in \mathrm{Lie}(D) \}$ of $T_q M$.
\begin{Def}
	The distribution $D$ is said to be \emph{bracket generating} if at any point $q \in M$ we have $\mathrm{Lie}(D)(q) = T_qM$.	
\end{Def}
In the rest of the paper, all distributions are supposed to be bracket generating. If $D$ is bracket generating then for any $q \in M$ there exists an integer $k$ such that $D^k(q) = T_q M$. The smallest integer with this property is called the \emph{nonholonomic order} (or simply \emph{the step}) of $D$ at $q$ and it is denoted by $r = r(q)$.
\begin{Def}
	A point $q_0 \in M$ is called \emph{regular} if, for every integer $k \geq 1$, the dimension $\dim D^k(q)$ is constant in a neighborhood of $q_0$.
\end{Def}

The \emph{weak derived flag of the distribution} $D$ at $q$ is the following filtration of vector spaces
	 \begin{equation}
	 \label{flag}
	D(q) = D^1(q) \subset D^2(q) \subset \dots \subset D^r(q) = T_qM.
	 \end{equation}	
For any positive integer $k$, we set $\dim D^k(q) = m_k(q)$. In particular, $m_1 = m$ and $m_r = n$. We call \emph{weights} at $q$ the  integers $w_1(q), \dots, w_{n}(q)$ defined by $w_i(q) = s, \text{ if } m_{s-1} < i \leq m_s$, where we set $m_0 = 0$.

\begin{Def}
A set of vector fields $\{X_1, \dots, X_n\}$ is called a \emph{frame of $TM$ adapted to $D$} at $q \in M$ if for any integer $i \in \{1,\dots,n\}$, the vector field $X_i$ belongs to $D^{w_i}$, and for any integer $k \in \{1,\dots ,r\}$, the vectors $X_1(q), \dots, X_{m_k}(q)$ form a basis of $D^k(q)$.
The \emph{structure coefficients} of the frame are the real-valued functions $c^k_{ij}$, $i,j,k \in \{1,\dots,n\}$, defined near $q$ by
$$
[X_i,X_j] = \sum_{k=1}^n c^k_{ij} X_k.
$$
\end{Def}
Such a frame can be constructed in the following way.
We start by choosing vector fields $X_1, \dots, X_m \in D^1$ whose values at $q$ form a basis of $D(q)$. Then we choose $m_2 - m$ vector fields $X_{m+1}, \dots, X_{m_2}$ among $\{ [X_i, X_j], \  1 \leq i,j \leq m \}$  whose values at $q$ form a basis of $D^2(q)$. Continuing in this way we get a set of vector fields $X_1, \dots, X_n$  such that $\mathrm{span} \{X_1(q), \dots, X_{m_k}(q)\} = D^k(q)$ for every integer $k\leq r$. In particular, $X_1(q), \dots, X_n(q)$ form a basis of $T_{q} M$.
Note that if $q$ is a regular point, then a frame adapted at $q$ is also adapted at any point near $q$.\medskip

Choosing now a Riemannian metric $g$ on $D$, we obtain a sub-Riemannian manifold $(M,D,g)$.
By abuse of notations we also say that $g$ is a \emph{sub-Riemannian metric} on $(M,D)$.
As mentioned in the introduction, the geodesics of $(M,D,g)$ are the projections on $M$ of the Pontryagin extremals associated with the minimization of energy. There exist two types of geodesics, the normal and abnormal ones. Abnormal geodesics depend only on the distribution $D$, not on $g$, hence they are of no use for the study of equivalence of metrics. Normal geodesics admit the following description.

For $q\in M$, we define a semi-norm on $T^*_q M$ by
$$
\| p \|_q = \max \left\{ \left< p, v\right> :  v \in D(q), \ g(q)(v,v) = 1 \right\}, \qquad p \in T^*_q M.
$$
The \emph{Hamiltonian} of the sub-Riemannian metric $g$ is the function $h: T^*M \rightarrow \R$ defined by
$$
h(q,p) = \frac{1}{2} \|  p \|_q^2, \qquad q\in M, \ p \in T^*_qM.
$$
\begin{Def}
	A \emph{normal extremal} is a trajectory $\lambda(\cdot)$ of the Hamiltonian vector field, i.e.\ $\lambda(t) = e^{t\vec{h}}\lambda_0$ for some $\lambda_0 \in T^*M$. A \emph{normal geodesic} is the projection $\gamma(t) = \pi (\lambda(t))$ of a normal extremal, where $\pi: T^*M\rightarrow M$ is the canonical projection.
\end{Def}

It is useful to give the expression of $\vec{h}$ in local coordinates. Fix a point $q_0 \in M$ and choose a frame $\{ X_1, \dots, X_n \}$ of $TM$ adapted to $D$ at $q_0$ such that $X_1, \dots, X_m$ is a $g$-orthonormal frame of $D$. At any point $q$ in a neighborhood $U$ of $q_0$, the basis $X_1(q), \dots, X_n(q)$ of $T_q M$ induces coordinates $(u_1,\dots,u_n)$  on $T^*_q M$ defined as $u_i(q,p) = \langle p, X_i(q) \rangle$. These coordinates in turn induce a basis $\partial_{u_1}, \dots, \partial_{u_n}$ of $T_{\lambda}(T^*_qM)$ for any $\lambda \in \pi^{-1}(q)$. For $i = 1, \dots, n$, we define the lift $Y_i$ of $X_i$ as the (local) vector field on $T^*M$ such that $\pi_* Y_i = X_i$ and $du_j (Y_i) = 0 \ \ \forall 1 \leq j \leq n$. The family of vector fields $\{ Y_1, \dots, Y_n, \partial_{u_1}, \dots, \partial_{u_n}\}$ obtained in this way is called a \emph{frame of $T(T^*M)$ adapted at} $q_0$. By a standard calculation, we obtain
\begin{equation}
\label{eq:vech}
h = \frac{1}{2}\sum_{i=1}^{m} u_i^2 \qquad \hbox{and} \qquad \vec{h}=\sum_{i=1}^{m} u_i Y_i + \sum_{i=1}^{m} \sum_{j,k=1}^{n} c^{k}_{ij} u_i u_k \partial_{u_j}.
\end{equation}

Note that a normal geodesic $\gamma(t) = \pi (e^{t\vec{h}}\lambda_0)$ satisfies $g(\dot{\gamma}(t),\dot{\gamma}(t)) = 2 h (\lambda(t)) = 2 h(\lambda_0)$, so the geodesic is arclength parameterized if $h(\lambda_0)=1/2$.

\subsection{Jacobi curves}
\label{jacsec}
As mentioned in the introduction the notion of Jacobi curve of a normal extremal is important for the considered equivalence problem.
This notion, introduced in \cite{agrgam1}, comes from the notion of Jacobi fields in Riemannian geometry. A Jacobi field is a vector field along a geodesic which carries information about minimizing properties of the geodesic. The Jacobi curve is a generalization of the space of Jacobi fields which can be defined in sub-Riemannian geometry.

For completeness we introduce Jacobi curves and all necessary related objects here, for more details we refer to \cite{jac1, li2, ref2}. Consider a sub-Riemannian manifold $(M,D,g)$ and a normal geodesic $\gamma(t) \in M, \ t \in [0,T]$. It is the projection on $M$ of an extremal $\lambda(t) = e^{t\vec{h}}\lambda$ for some $\lambda \in T^*M \text{ such that } \pi(\lambda) = \gamma(0)$. The $2n$-dimensional space $T_{\lambda}(T^*M)$ endowed with the natural symplectic form $\sigma_{\lambda}(\cdot,\cdot)$ is a symplectic vector space. A Lagrangian subspace of this symplectic space is a vector space of dimension $n$ which annihilates the symplectic form. We denote by $ \mathcal{V}_{\lambda(t)}$ the vertical subspace $T_{\lambda(t)}(T_{\gamma(t)}^*M)$  of   $T_{\lambda(t)}(T^*M)$, it is vertical in the sense that $\pi_* (\mathcal{V}_{\lambda(t)}) = 0$. Now we can define the Jacobi curve  associated with the normal geodesic $\gamma(t)$.
\begin{Def}
\label{def:jacobi}
	For $\lambda \in T^*M$, we define the \emph{Jacobi curve} $J_{\lambda}(\cdot)$ as the curve of Lagrangian subspaces of $T_{\lambda}(T^*M)$ given by
	$$ J_{\lambda}(t) = e^{-t\vec{h}}_*\mathcal{V}_{\lambda(t)}, \qquad t \in [0, T].$$
\end{Def}
\label{def:extension}
We introduce the extensions of a Jacobi curve as an analogue to the Taylor expansions at different orders of a smooth curve.
\begin{Def}
For an integer $i \geq 0$, the \emph{$i$th extension} of the Jacobi curve $J_{\lambda}(\cdot)$ is defined as
$$
J_{\lambda}^{(i)} =  \mathrm{span} \left\{  \frac{d^j}{dt^j}l(0) \ : \ l(s) \in J_{\lambda}(s)\ \forall s \in [0, T], \ l(\cdot) \text{ smooth }, \ 0\leq j \leq i \right\}.
$$
\end{Def}
In other words, these spaces are spanned by all the directions generated by derivatives at $t=0$ of the standard curves lying in the Jacobi curve. By definition, $J^{(i)}_{\lambda} \subset J^{(i+1)}_{\lambda} \subset T_{\lambda}(T^* M)$, so it is possible to define a flag of these spaces.
\begin{Def}
The \emph{flag of the Jacobi curve} $J_{\lambda}(\cdot)$ is defined as
	 \begin{equation}
	 \label{flag.jc}
	 J_{\lambda} = J_{\lambda}^{(0)} \subset J_{\lambda}^{(1)} \subset \dots \subset T_{\lambda}(T^*M).
	 \end{equation}
\end{Def}

In an adapted frame of $T(T^*M)$, the Jacobi curves can be obtained from iterations of Lie brackets by $\vec{h}$. Let us remind first that, for a positive integer $l$ and a pair of vector fields $X,Y$, the notation $(\mathrm{ad}X)^l Y$ stands for $\underbrace{ [ X, \dots, [ X}_\text{$l$ times} ,Y ] ]$.	

\begin{Lemma}
\label{le:jacobi}
Let $q=\pi(\lambda)$. In an adapted frame $\{ Y_1, \dots, Y_n, \partial_{u_1}, \dots, \partial_{u_n}\}$ of $T(T^*M)$ at $q$, the extensions of the Jacobi curve take the following form:
\begin{equation*}
	 \begin{aligned}
	& J_{\lambda}^{(0)} =  \Big\{ v \in T_{\lambda}(T^* M) \ : \ \pi_* v = 0 \Big\}, \\
	& J_{\lambda}^{(1)} =  \Big\{ v \in T_{\lambda}(T^* M) \ : \ \pi_* v \in D \Big\} = J^{(0)}_{\lambda} + \mathrm{span} \Big\{Y_1(\lambda), Y_2(\lambda), \dots, Y_m(\lambda)\Big\},\\
	& J_{\lambda}^{(2)} =  J^{(1)}_{\lambda} + \mathrm{span}\Big\{[\vec{h}, Y_1](\lambda), \dots, [\vec{h}, Y_m](\lambda)\Big\}, \\
	\vdots & \\
	& J_{\lambda}^{(k)} =  J^{(k-1)}_{\lambda} + \mathrm{span}\Big\{(\mathrm{ad}\vec{h})^{k-1} Y_1(\lambda), \dots, (\mathrm{ad}\vec{h})^{k-1} Y_m(\lambda) \Big\}.
	 \end{aligned}
\end{equation*}
\end{Lemma}

\begin{proof}
Let $v \in J_{\lambda}^{(k)}$, for some integer $k \geq 0$. By definition,  $v = \frac{d^s}{dt^s}l(0)$ where $l(\cdot)$ is a curve with $l(t) \in J_{\lambda}(t)$ for any $t\in [0, T]$, and $s \leq k$ is an integer. Then there exists a vertical vector field $Y$ on $T^*M$ (i.e.\ $\pi_*Y=0$) such that, for any $t \in [0,T]$,
$$
l(t) =  e^{-t\vec{h}}_* Y(\lambda(t)),
$$
and $v$ writes as
\begin{equation}
\left. v = \dfrac{d^s}{d t^s}e^{-t\vec{h}}_* Y(e^{t\vec{h}}\lambda) \right|_{t=0} =  (\mathrm{ad}\vec{h})^{s} Y(\lambda).
\end{equation}

As $Y$ is a vertical vector field, in the adapted frame $\{ Y_1, \dots, Y_n, \partial_{u_1}, \dots, \partial_{u_n}\}$ it can be written as $Y = \sum_{i=1}^{n} a_i \partial_{u_i}$. Using the expression \eqref{eq:vech} of $\vec{h}$ in this frame, we obtain
\begin{align*}
[\vec{h},Y]  & = \left[\displaystyle\sum_{i=1}^{m} u_i Y_i + \displaystyle \sum_{i=1}^{m} \sum_{j,k=1}^{n} c^{k}_{ij} u_i u_k \partial_{u_j}, \displaystyle\sum_{i=1}^{n} a_i \partial_{u_i}\right] \\
& = \sum_{i=1}^{m} a_i Y_i  \mod  \mathrm{span} \{ \partial_{u_1}, \dots, \partial_{u_n}\}.
\end{align*}
By iteration, we get
\begin{eqnarray*}
(\mathrm{ad}\vec{h})^{2} Y  &=& \sum_{i=1}^{m} a_i \, [\vec{h},Y_i]  \mod  \mathrm{span} \{ \partial_{u_1} \dots, \partial_{u_n}, Y_1, \dots, Y_m\} \\
& \vdots &  \\
(\mathrm{ad}\vec{h})^{s} Y   &=& \sum_{i=1}^{m} a_i \, (\mathrm{ad}\vec{h})^{s-1}Y_i  \\
& & \ \hfill \mod  \mathrm{span} \{ \partial_{u_1} \dots, \partial_{u_n}, Y_i,  \dots, (\mathrm{ad}\vec{h})^{s-2}Y_i, \ i=1,\dots,m\},
\end{eqnarray*}
which proves the result.	
\end{proof}

\subsection{Ample geodesics}
\label{se:ample}

Note that the dimension of the spaces $J^{(k)}_{\lambda}$ for $|k|>1$ may depend on $\lambda$ in general. Following \cite{ref2}, we distinguish the geodesics corresponding to the extensions of maximal dimension.
\begin{Def}
\label{def:ample}
The normal geodesic $\gamma(t)=\pi(e^{t\vec{h}}\lambda)$ is said to be \emph{ample at $t=0$}  if there exists an integer $k_0$ such that
$$
\dim (J^{(k_0)}_{\lambda}) = 2n.
$$
In that case we say that $\lambda$ \emph{is ample with respect to the metric $g$}.
\end{Def}
Notice that if a geodesic is ample at $t=0$, then it is not abnormal on any small enough interval $[0,\varepsilon]$ (see \cite[Prop.\ 3.6]{ref2}).
It appears that normal geodesic are generically ample in the following sense.
\begin{Theorem}[\cite{ref2}, Proposition 5.23]
\label{th:Agr}
For any $q \in M$, the set of ample covectors $\lambda \in \pi^{-1}(q)$ is an open and dense (and hence non empty) subset of $T^*_qM$.
\end{Theorem}
Note that, from the cited result, the set of ample covectors not only is open and dense but it has also full measure in $T^*_qM$ since it contains a Zarisky set.
\begin{rk}
Two proportional covectors $\lambda$ and $c \lambda$, $c>0$, define the same geodesic,  up to time reparameterization, and the corresponding extensions of Jacobi curves $J^{(k)}_{\lambda}$ and $J^{(k)}_{c \lambda}$ have the same dimension. As a consequence, the statement of Theorem~\ref{th:Agr} also holds in $h^{-1}(1/2)$. Namely, the set of ample covectors is open and dense in $\pi^{-1}(q) \cap h^{-1}(1/2)$.
\end{rk}

Ample geodesics play a crucial role in the study of equivalence of metrics because they are the geodesics characterized by their jets. Let us precise this fact.
Fix a nonnegative integer $k$. For a given curve $\gamma: I\rightarrow M$, $I\subset \mathbb R$, denote by
$j_{t_0}^{k}\gamma$ the $k$-jet of $\gamma$ at the point $t_0$. Given $q\in M$, we denote by $J^k_q (g)$ the space of $k$-jets at $t=0$ of the normal geodesics of $g$ issued from $q$ and parameterized by arclength. We set $J^k(g)=\displaystyle{\bigsqcup_{q\in U} J^k_q(g)}$.
	
Define the maps $P^{k}: H\mapsto
	J^k(g)$, where $H=h^{-1}(1/2)$, by
	 \begin{equation*}
	 P^{k}(\lambda)= j_0^k\pi(e^{t\vec {h}}\lambda).
	 \end{equation*}
The properties of the map $P^k$ near a point $\lambda$ can be described in terms of the $k$th extension $J^{(k)}_{\lambda}$ of the Jacobi curve. Let us denote by $\left(J^{(k)}_{\lambda}\right)^{\angle}$ the skew-symmetric complement of $J^{(k)}_{\lambda}$ with respect to the symplectic form $\sigma_\lambda$ on $T_\lambda(T^*M)$, i.e.,
\begin{equation}
\left(J^{(k)}_{\lambda}\right)^{\angle} = \left\{ v \in T_\lambda(T^*M) \ : \ \sigma_\lambda(v,w)=0 \quad \forall w \in J^{(k)}_{\lambda} \right\}.
\end{equation}
	
\begin{Lemma}
\label{lem}
For any integer $k \geq 0$, the kernel of the differential of the map  $P^k$ at a point $\lambda$ satisfies
\begin{equation}
		 \label{kernel}
		\ker d P^k(\lambda) \subset \left(J^{(k)}_{\lambda}\right)^{\angle}.
\end{equation}
\end{Lemma}

\begin{proof}
Let $\lambda \in H$ and fix a canonical system of coordinates on $T^*M$ near $\lambda$. In particular, in such coordinates $\pi$ is a linear projection.

Let $v$ be a vector in $\ker d P^k(\lambda)$. Then there exists a curve $s \mapsto \lambda_s$ in $H$ such that $\lambda_0=\lambda$, $\left. \frac{d\lambda_s}{ds}\right|_{s=0}  = v$, and the following equalities holds in the fixed coordinate system:
\begin{equation}
	\label{eq:e*}
	\left. \frac{\partial^{l+1}}{\partial t^l \partial s} \left(\pi \circ e^{t \vec{h}} \lambda_s \right) \right|_{(t,s)=(0,0)} = \left. d \pi\circ \frac{d^l}{d t^l} \left(  e_*^{t\vec{h}} v\right)\right|_{t=0}= 0 \qquad \forall \ 0 \leq l \leq k.
\end{equation}

Consider now $w \in  J^{(k)}_{\lambda}$. Then there exists an integer $j$, $0 \leq j \leq k$, and a vertical vector field $Y$ (i.e., $d\pi \circ Y=0$) on $T^*M$ such that $w$ writes as
\begin{equation}
 w = \dfrac{d^j}{d t^j} \left. \left( e^{-t\vec{h}}_* Y(e^{t\vec{h}}\lambda) \right) \right|_{t=0}.
\end{equation}
We have
\begin{align}\label{startchain}
\sigma_\lambda (v,w) & = \sigma_\lambda \left(v, \dfrac{d^j}{d t^j} \left. \left( e^{-t\vec{h}}_* Y(e^{t\vec{h}}\lambda) \right) \right|_{t=0} \right), \\
& = \dfrac{d^j}{d t^j} \left. \left( \sigma_{\lambda} \left(v, e^{-t\vec{h}}_* Y(e^{t\vec{h}}\lambda) \right) \right) \right|_{t=0}.
\end{align}
The last equality holds, because we work with the fixed bilinear form $\sigma_\lambda$ on the given vector space $T_\lambda T^*M$.

Using now that $e^{t\vec{h}}$ is a symplectomorphism, we obtain
\begin{align}
\sigma_\lambda (v,w)  = \dfrac{d^j}{d t^j} \left. \left( \sigma_{e^{t\vec{h}}\lambda} \left(e^{t\vec{h}}_*v, Y(e^{t\vec{h}}\lambda) \right) \right) \right|_{t=0}
\end{align}
So far, all equalities starting  from \eqref{startchain} were obtained in a coordinate-free manner. Now use again the fixed canonical coordinate system
on $T^*M$ near $\lambda$. In these coordinates, the form $\sigma$ is in the Darboux form. In particular, the coefficients of this form are constants. Therefore,
\begin{align}
~&\sigma_\lambda (v,w) =  \sum_{l=1}^j  \begin{pmatrix}
   j \\  l \\   \end{pmatrix}  \sigma_{\lambda} (v_l,w_l), \\
~& \hbox{where} \quad v_l= \dfrac{d^l}{d t^l} \left. \left( e^{t\vec{h}}_*v \right) \right|_{t=0} \quad  \hbox{and} \quad w_l= \dfrac{d^{j-l}}{d t^{j-l}} \left. \left( Y(e^{t\vec{h}}\lambda) \right) \right|_{t=0} .
\end{align}
By \eqref{eq:e*}, every  vector $v_l$ is vertical. The vectors $w_l$ in the chosen coordinate system are vertical  as well since the vector field $Y$ is vertical. As a consequence, $\sigma_{\lambda} (v_l,w_l)=0$, which implies $\sigma_\lambda (v,w)=0$. This completes the proof.
\end{proof}

\begin{rk}
When the Jacobi curve is equiregular (i.e., the dimensions $\dim J^{(-k)}_{\lambda(t)}$, $k \in \N$,  are constant for $t$ close to $0$), the skew-symmetric complement of the $k$th extension is equal to  the $k$th contractions $J^{(-k)}_{\lambda}$ of the Jacobi curve (see \cite[Lemma 1]{li2}). In that case we can show the equality $\ker d P^k(\lambda) = J^{(-k)}_{\lambda}$.
\end{rk}

Since $\dim \left(J^{(k)}_{\lambda}\right)^{\angle} = 2n - \dim J^{(k)}_{\lambda}$, we get as a corollary of Lemma~\ref{lem}  that ample geodesics are characterized locally by their $k$-jets for $k$ large enough.

\begin{Coro}
\label{le:immersion}
Let $\lambda \in T^*M$ be ample. Then there exists an integer $k_0$ such that the map $P^{k_0}$ is an immersion at $\lambda$.
\end{Coro}

\section{Orbital diffeomorphism}
\label{se:orb_diffeo}

Projectively or affinely equivalent metrics have the same geodesics, up to the appropriate reparameterization. But do they have the same (normal) Hamiltonian vector field, up to an appropriate transformation? In particular, is it possible to recover the Hamiltonian vector field of a metric from the knowledge of the geodesics? We will see that both questions have a positive answer near ample geodesics, for which the covector can be obtained from the jets of the geodesics (see Corollary~\ref{le:immersion}).

\subsection{Orbital diffeomorphism on ample geodesics}
\label{se:orbital}

Fix a manifold $M$ and a bracket generating distribution $D$ on $M$, and consider two sub-Riemannian metrics $g_1$ and $g_2$ on $D$. We denote by $h_1$ and $h_2$ the respective sub-Riemannian Hamiltonians of $g_1$ and $g_2$,  and by $H_1=h_1^{-1}(1/2)$ and $H_2=h_2^{-1}(1/2)$ the respective $\frac{1}{2}$-level sets of these Hamiltonians.

\begin{Def}
We say that $\vec h_1$ and  $\vec h_2$ are \emph{orbitally diffeomorphic}
on an open subset $V_1$ of $H_1$ if there exists an open subset $V_2$ of $H_2$ and a diffeomorphism $\Phi: V_1 \rightarrow V_2$ such that $\Phi$ is fiber-preserving, i.e.\  $\pi (\Phi(\lambda)) = \pi (\lambda)$, and $\Phi$ sends the integral curves of $\vec h_1$ to the reparameterized integral curves of $\vec h_2$, i.e., there exists a smooth function $s=s(\lambda,t)$ with $s(\lambda,0)=0$ such that $\Phi\bigl(e^{t\vec h_1}\lambda\bigr) =e^{s\vec h_2} \bigl(\Phi(\lambda)\bigr)$ for all $\lambda \in V_1$ and $t\in \mathbb R$ for which $e^{t\vec h_1}\lambda$ is well defined. Equivalently, there exists a smooth function $\aa(\lambda)$ such that
\begin{equation}
\label{orbeq}
d\Phi \circ \vec{h}_1(\lambda) = \aa(\lambda) \vec{h}_2(\Phi(\lambda)).
\end{equation}
The map $\Phi$ is called an \emph{orbital diffeomorphism} between the extremal flows of $g_1$ and  $g_2$.
\end{Def}

\begin{rk}
\label{rem:H1}
In the definition above, the orbital diffeomorphism $\Phi$ is defined as a mapping from $H_1$ to $H_2$. However it can be easily extended as a mapping $\bar \Phi$ from $T^*M \setminus h_1^{-1}(0)$ to itself by rescaling, i.e.,
$$
\bar \Phi (\lambda) = \sqrt{2 h_1(\lambda)} \Phi \left(\frac{\lambda}{\sqrt{2 h_1(\lambda)}}\right).
$$
This mapping sends the level sets $h_1^{-1}(C^2/2)$ of $h_1$ to the level sets $h_2^{-1}(C^2/2)$ of $h_2$, and the integral curves of $\vec{h}_1$ to the ones of $\vec{h}_2$. In particular \eqref{orbeq} holds with a function $\bar \aa (\lambda) = \aa (\lambda / \sqrt{2 h_1(\lambda)})$. 
\end{rk}

\begin{Prop}
\label{th:proj.to.orb}
If $\vec{h}_1$ and $\vec{h}_2$ are orbitally diffeomorphic on a neighborhood of $H_1 \cap \pi^{-1}(q_0)$, then $g_1, g_2$ are projectively equivalent at $q_0$. If in addition the function $\aa(\lambda)$ in  \eqref{orbeq} satisfies $\vec{h}_1(\aa) = 0$, then $g_1, g_2$ are affinely equivalent.
\end{Prop}
\begin{proof}
The first property is obvious. Indeed, if $\vec{h}_1$ and $\vec{h}_2$ are orbitally diffeomorphic, then the relation $\Phi\bigl(e^{t\vec h_1}\lambda\bigr) =e^{s\vec h_2} \bigl(\Phi(\lambda)\bigr)$ implies that any normal geodesics of $g_2$ near $q_0$ satisfies
$$
\pi (e^{s\vec h_2} \lambda)= \pi \circ \Phi \bigl(e^{t\vec h_1}\bigl(\Phi^{-1}(\lambda)\bigr)\bigr)=\pi \circ e^{t\vec h_1}\bigl(\Phi^{-1}(\lambda)\bigr),
$$
and thus coincides with a normal geodesic of $g_1$. Since on the other hand abnormal geodesics always coincide, the metrics $g_1,g_2$ have the same geodesics near $q_0$, and thus are projectively equivalent at $q_0$.

Note that $s=s(\lambda,t)$ is the reparameterization of time and that $\aa(\lambda)= \frac{d s}{dt}(\lambda,0)$. If $\vec{h}_1(\aa) = 0$, then the function $\aa$ is constant along the geodesics and the time-reparameterization is affine, which implies that the metrics are affinely equivalent.
\end{proof}

We have actually a kind of converse statement near ample geodesics.

\begin{Prop}
\label{th:orb.to.proj}
Assume that the sub-Riemannian metrics $g_1$ and $g_2$ are projectively equivalent at $q_0$. Then, for any covector $\lambda_1 \in H_1 \cap \pi^{-1}(q_0)$ ample with respect to $g_1$,  $\vec{h}_1$ and $\vec{h}_2$ are orbitally diffeomorphic on a neighborhood $V_1$ of $\lambda_1$ in $T^*M$.

If moreover $g_1$ and $g_2$ are affinely equivalent at $q_0$, then the function $\aa(\lambda)$ in~\eqref{orbeq} satisfies $\vec{h}_1(\aa)=0$. \end{Prop}

\begin{proof}
Assume that $U$ is a neighborhood of $q_0$ such that $g_1$ and $g_2$ have the same geodesics in $U$, up to a reparameterization. Then $g_1$ and $g_2$ have the same ample geodesics in $U$, up to a reparameterization. Indeed, a geodesic $\gamma(t)=\pi(e^{t\vec{h}_1}\lambda)$ of $g_1$ which is ample at $t=0$ is a geodesics of $g_2$ as well by assumption, and moreover a normal one since ample geodesics are not abnormal. The conclusion follows then from the fact that being ample at $t=0$ with respect to $g_1$ is a property of $\gamma(t)=\pi(e^{t\vec{h}_1}\lambda)$ as an admissible curve (see \cite[Proposition 6.15]{ref2}), and does not depend neither on the time parameterization nor on the Hamiltonian vector field.

Fix a nonnegative integer $k$. As in Subsection~\ref{se:ample}, for $q\in U$ and $i=1,2$, we denote by $J^k_q (g_i)$ the space of $k$-jets at $t=0$ of the normal geodesics of $g_i$ issued from $q$ and parameterized by arclength parameter with respect to the sub-Riemannian metric $g_i$. We set $J^k(g_i)=\displaystyle{\bigsqcup_{q\in U} J^k_q(g_i)}$ and we define $P_i^{k}:H_i\mapsto
	J^k(g_i)$ by
	 \begin{equation*}
	 P_i^{k}(\lambda) = j_0^k\pi(e^{t\vec {h}_i}\lambda).
	 \end{equation*}

Let $\lambda_1 \in H_1 \cap \pi^{-1}(q_0)$ be an ample covector with respect to $g_1$. Then by Corollary~\ref{le:immersion} for a large enough integer $k$ there exists a neighborhood $V_1$ of $\lambda_1$ in $H_1$ such that the map $P_1^{k}|_{V_1}$ is a diffeomorphism on its image. Up to reducing $V_1$ we assume that $\pi(V_1) \subset U$ and that every $\lambda \in V_1$ is ample. As a consequence, every geodesic $\pi(e^{t\vec{h}_1}\lambda)$ with $\lambda \in V_1$ is an ample geodesic with respect to $g_2$.

Let $\lambda_2 \in \pi^{-1}(q_0) \cap H_2$ be the covector such that the curves $\pi(e^{t\vec{h}_1}\lambda_1)$ and $\pi(e^{t\vec{h}_2}\lambda_2)$ coincide up to time reparameterization ($\lambda_2$ is unique since an ample geodesic is not abnormal).  Since $\lambda_2$ is ample with respect to $g_2$,  the same argument as above shows that there exists a neighborhood $V_2$ of $\lambda_2$ in $H_2$ such that $P_2^{k}|_{V_2}$ is a diffeomorphism on its image. Up to reducing $V_1$ and $V_2$ if necessary, the reparameterization of the geodesics from the arclength parameter with respect to $g_1$ to the arclength parameter with respect to $g_2$ induces naturally a diffeomorphism $\Psi_k: P_1^{k}(V_1) \subset J^k(g_1) \to P_2^{k}(V_2) \subset J^k(g_2)$. Thus the map $\Phi$ which completes the following
diagram into a commutative one,
	 \begin{equation}
	 \label{commdiag}
	 \begin{diagram}
	V_1 \subset H_1 & \rTo^\Phi & V_2 \subset H_2 \\
	 \dTo^{P_1^k} & \qquad \qquad & \dTo_{P_2^k} \\
	 P_1^{k}(V_1) \subset J^k(g_1)& \rTo^{\Psi_k} & P_2^{k}(V_2) \subset J^k(g_2)
	 \end{diagram}
	 \end{equation}
defines an orbital diffeomorphism between $V_1$ and $V_2$. Due to Remark~\ref{rem:H1} $V_1$ can be extended to a neighborhood of $\lambda_1$ in the whole $T^*M$, which completes the proof of the first part of the proposition.

Assume now that $g_1$ and $g_2$ are affinely equivalent at $q_0$. Then the map $\Phi$ satisfies $\Phi\bigl(e^{t\vec h_1}\lambda\bigr) =e^{s\vec h_2} \bigl(\Phi(\lambda)\bigr)$, where $s=s(\lambda,t)$ is the reparameterization of time and $\frac{d s}{dt}(\lambda,0) =\aa(\lambda)$. Since $g_1$ and $g_2$ are affinely equivalent, $s(\lambda,t)$ must be affine with respect to $t$, which implies that $\aa(e^{t\vec h_1}\lambda)$ is constant, and thus $\vec{h}_1(\aa)=0$.
\end{proof}

\begin{rk}
We have seen in the proof just above that two projectively equivalent metrics have the same set of ample geodesics. In the same way, one can prove that they have the same set of strictly normal geodesics. However we can not affirm that they have the same normal geodesics: a geodesic could be both normal and abnormal for $g_1$ and only abnormal for $g_2$.
\end{rk}

\subsection
{Fundamental algebraic system}
\label{se:coord_orbdiffeo}
Let $M$ be a smooth manifold and $D$ be a bracket generating distribution on $M$. Let us fix two sub-Riemannian metrics $g_1, g_2$ on ($M,D$).

\begin{Def}
The \emph{transition operator} at a point $q \in M$ of the pair of metrics $(g_1, g_2)$ is the linear operator $S_q : D_q \rightarrow D_q$ such that $g_1(q)(S_q v_1, v_2) = g_2(q)(v_1, v_2)$ for any $v_1, v_2 \in D_q$.
\end{Def}
Obviously $S_{q}$ is a positive $g_1$-self-adjoint operator and its eigenvalues $\alpha^2_1(q)$, \dots, $\alpha^2_m(q)$ are positive real numbers (we choose $\alpha_1(q), \dots, \alpha_m(q)$ as positive numbers as well). Denote by $N(q)$ the number of distinct eigenvalues of $S_q$.
\begin{Def}
A point $q \in M$ is said to be \emph{stable with respect to}  $g_1, g_2$ if $q$ is a regular point and $N(\cdot)$ is constant in some neighborhood of $q$.
\end{Def}
The set of regular points and the set of points where $N(q)$ is locally constant are both open and dense in $M$, and so is the set of stable points.

Let us fix a stable point $q_0$. In a neighborhood $U$ of $q_0$ we can choose a $g_1$-orthonormal frame $X_1, \dots, X_m$ of $D$ whose values at any $q \in U$ diagonalizes $S_q$, i.e.\ $X_1(q), \dots, X_m(q)$ are eigenvectors of $S_q$ associated with the eigenvalues  $\alpha^2_1(q), \dots, \alpha^2_m(q)$ respectively. Note that $\frac{1}{\alpha_1}X_1, \dots, \frac{1}{\alpha_m}X_m$ form a $g_2$-orthonormal frame of $D$.  We then complete $X_1, \dots, X_m$ into a frame $\{X_1, \dots, X_n\}$ of $TM$ adapted to $D$ at $q_0$. We call such a set of vector fields $\{X_1, \dots, X_n\}$  a \emph{(local) frame adapted to the (ordered) pair of metrics} $(g_1, g_2)$.

Let $u = (u_1, \dots, u_n)$ be the coordinates on the fibers $T_q^*M$  induced by this frame, i.e.\ $u_i(q,p) = \langle p, X_i(q) \rangle$. The Hamiltonian functions $h_1$ and $h_2$ associated respectively with $g_1$ and $g_2$ write as
$$
h_1 = \frac{1}{2}\sum_{i=1}^{m} u_i^2, \quad h_2 = \frac{1}{2}\sum_{i=1}^{m} \frac{u_i^2}{\alpha_i^2}.
$$
In the corresponding frame $\{ Y_1, \dots, Y_n, \partial_{u_1}, \dots, \partial_{u_n}\}$ of $T(T^*M)$, $\vec{h}_1$ has the form \eqref{eq:vech}, i.e.,
\begin{equation}
\label{eq:h1}
\vec{h}_1=\sum_{i=1}^{m} u_i Y_i + \sum_{i=1}^{m} \sum_{j,k=1}^{n} c^{k}_{ij} u_i u_k \partial_{u_j},
\end{equation}
and a simple computation gives
\begin{equation}
\label{eq:h2}
\vec{h}_2=\sum_{i=1}^{m} \frac{u_i}{\alpha_i^2} Y_i + \sum_{i=1}^{m} \sum_{j,k=1}^{n} \frac{c^{k}_{ij}}{\alpha_i^2} u_i u_k \partial_{u_j} - \sum_{i=1}^{m} \sum_{j=1}^{n} \frac{1}{\alpha_i} X_j(\frac{1}{\alpha_i}) u_i^2 \partial_{u_j}.
\end{equation}

Assume now that $\vec{h}_1$ and $\vec{h}_2$ are orbitally diffeomorphic near $\lambda_0\in H_1\cap \pi^{-1}(q_0)$ and let $\Phi$ be the corresponding orbital diffeomorphism. Following Remark~\ref{rem:H1}, we assume that $\Phi$ is defined on a neighborhood $V$ of $\lambda_0$ in the whole $T^*M$. Let us denote by $\Phi_i, \ i= 1, \dots, n,$ the coordinates $u_i$ of $\Phi$ on the fiber, i.e.\ $u\circ \Phi(\lambda) = (\Phi_1(\lambda), \Phi_2(\lambda), \dots, \Phi_n(\lambda))$.

Using \eqref{eq:h1} and \eqref{eq:h2}, we can write in coordinates the identity \eqref{orbeq}, i.e.\ $d\Phi \circ \vec{h}_1(\lambda)= \aa (\lambda) \vec{h}_2(\Phi(\lambda))$,  and deduce from there some conditions on the coordinates $\Phi_i$. This computation has been made in \cite{Z}, we just give the result here (our equations look a bit different than the ones of \cite{Z} because we use the structure coefficients $c_{ij}^k$ here instead of the $\bar c_{ji}^k$).

\begin{Lemma}[\cite{Z}, Lemmas 1 and 2]
\label{prop}
A smooth fiber-preserving map $\Phi$ from an open subset $V_1$ of $H_1$ to an open subset $V_2$ of $H_2$  satisfies \eqref{orbeq} if and only if the following conditions are satisfied:
\begin{itemize}
\item the function $\aa(\lambda)$ is given by
\begin{equation}
\label{eq:aa}
    \aa = \sqrt{\frac{\alpha_1^2 u_1^2 + \cdots + \alpha_m^2 u_m^2}{ u_1^2 + \cdots + u_m^2}},
\end{equation}
		 \item for $k = 1, \dots, m$,  	
\begin{equation}
\label{eq:first.m.coord}
\Phi_k = \frac{\alpha^2_k u_k}{\aa},
\end{equation}
\item for $j = 1, \dots, m$,
		 \begin{equation}
		 \label{eq:220}
		 \sum_{k= m+1}^{n} q_{jk} \Phi_{k} = \frac{R_j}{\aa},
		 \end{equation}
where $q_{jk} = \sum_{i=1}^m c_{ij}^k u_i$ and
\begin{multline}
\label{Rj}
R_j =\vec{h}_1(\alpha_j^2)u_j + \alpha_j^2 \vec{h}_1(u_j) - \frac{1}{2}\alpha_j^2 u_j \frac{\vec{h}_1(\aa^2)}{\aa^2} \\
 - \frac{1}{2} \sum_{i=1}^m X_j(\alpha_i^2) u_i^2 - \sum_{1 \leq i,k \leq m} c^k_{ij} \alpha_k^2 u_i u_k,
\end{multline}
\item for $k = m+1, \dots , n$,
		 \begin{equation}
		 \label{eq:221}
		 \vec{h}_1(\Phi_k) = \sum_{l= m+1}^{n} q_{kl} \Phi_{l} +\frac{1}{\aa} \sum_{i=1}^{m} u_i \left(\alpha_i^2 q_{ki} + \frac{X_k(\alpha_i^2)}{2} u_i\right).
		 \end{equation}
\end{itemize}
\end{Lemma}

\begin{rk}
The spectral size $N$ is equal to $1$ if and only if $g_2$ is conformal to $g_1$ near $q_0$. In that case $g_2=\aa^2 g_1$ and $\alpha_1=\cdots = \alpha_m= \aa$. In particular the function $\aa$ does not depend on $u$, i.e.\ $\alpha(\lambda)$ depends only on $\pi(\lambda)$.
\end{rk}

This lemma gives directly the values of the first $m$ components of $\Phi$. The difficulty now is to find the other components from \eqref{eq:220} and \eqref{eq:221}. It is more convenient to replace the differential equations \eqref{eq:221} by infinitely many linear algebraic equations, forming the \emph{fundamental algebraic system} as described  by the following proposition.
%
\begin{Prop}
	 \label{Aprop}
	Let $\Phi$ be an orbital diffeomorphism between the extremal flows of $g_1$ and $g_2$ with coordinates $(\Phi_1, \dots, \Phi_n)$. Set $\widetilde{\Phi}=(\Phi_{m+1}, \dots, \Phi_n)$. Then $\widetilde{\Phi}$ satisfies a linear system of equations,
	 \begin{equation}
	 \label{A.phi.B}
	A \widetilde{\Phi} = b,
	 \end{equation}
where $A$ is a matrix with $(n-m)$ columns and an infinite number of rows, and $b$ is a column vector with an infinite number of rows. These infinite matrices can be decomposed in layers of $m$ rows as
\begin{equation}
\label{eq:A_and_b}
		A = \left(
		 \begin{array}{c}
		A^1 \\
		A^2 \\
		 \vdots \\
		A^{s} \\
		 \vdots \\
		 \end{array}
		 \right)
		 \qquad \hbox{and}
		 \qquad b= \left(
		 \begin{array}{c}
		b^1 \\
		b^2 \\
		 \vdots \\
		b^{s} \\
		 \vdots \\
		 \end{array}
		 \right),
\end{equation}
where the coefficients $a^s_{jk}$ of the $(m \times (n-m))$ matrix $A^s$, $s \in \N$,  are defined by induction as
		 \begin{equation}
		 \label{elem.A}
		 \left\{
		 \begin{aligned}
		& a^1_{j,k} = q_{jk}, \qquad  &1 \leq j \leq m, \ m < k  \leq n,\\
		& a^{s+1}_{j, k} = \vec{h}_1(a^s_{j,k}) + \sum_{l = m+1}^{n}a^s_{j,l} q_{l k}, \qquad  &1 \leq j \leq m, \ m < k  \leq n,
		 \end{aligned}
		 \right.
		 \end{equation}
(note that the columns of $A$ are numbered from $m+1$ to $n$ according to the indices of $\widetilde{\Phi}$) and the coefficients $b^{s}_j$, $1 \leq j \leq m$, of the vector $b^s \in \R^m$ are defined by
\begin{equation}
		 \label{term.of.b.rec.form}
		 \left\{
\begin{aligned}
		& b^{1}_{j} = \ \frac{R_j}{\aa}, \\
		& b^{s+1}_{j} = \ \vec{h}_1(b^s_{j}) - \frac{1}{\aa}\sum_{k = m+1}^{n} a^s_{j,k} \sum_{i=1}^{m} u_i \left(\alpha_i^2 q_{ki} + \frac{X_k(\alpha_i^2)}{2} u_i\right).  \\
		 \end{aligned}
		 \right.
\end{equation}
\end{Prop}	

Note that $A$ is a function of $u$ and this function only depends on the choice of the local frame $\{X_1, \dots, X_n\}$. On the other hand the vector-valued function $b$ depends on $\{X_1, \dots, X_n\}$ and on $\{\alpha_1, \dots, \alpha_m\}$.

\begin{proof}
We have to prove that, for every $s \in \N$, the coordinates $\widetilde{\Phi}$ satisfy
\begin{equation}
		 \label{eq:As}
		A^s \widetilde{\Phi} = b^s.
\end{equation}
Observe first that \eqref{eq:220} is exactly $A^1 \widetilde{\Phi} = b^1$, so \eqref{eq:As} holds for $s=1$. Assume by induction that it holds for a given $s$. Thus we have, for $j=1,\dots,m$,
$$
\sum_{k=m+1}^{n}a^s_{j,k} \Phi_{k} = b^s_j.
$$
Taking the Lie derivative of these expressions by $\vec{h}_1$, we get
$$
 \sum_{k=m+1}^{n} \vec{h}_1(a^s_{j,k})\Phi_{k} + \sum_{k=m+1}^{n}a^s_{j,k} \vec{h}_1(\Phi_{k}) = \vec{h}_1(b^s_j).
$$
Replacing every term $\vec{h}_1(\Phi_{k})$ by its expression in \eqref{eq:221} and reorganizing, we obtain a new linear equation,

\begin{multline}
		 \label{eq:elements.a-b}
		 \sum_{k = m+1}^{n}\left(\vec{h}_1(a^s_{j,k}) + \sum_{l = m+1}^{n}a^s_{j,l}q_{lk} \right) \Phi_{k} =  \\
\vec{h}_1(b^s_j) - \frac{1}{\aa} \sum_{k=m+1}^{n} a^s_{j,k} \sum_{i=1}^{m} u_i \left(\alpha_i^2 q_{ki} + \frac{X_k(\alpha_i^2)}{2} u_i\right),
\end{multline}
which is exactly the $j$th row of $A^{s+1} \widetilde{\Phi} = b^{s+1}$. This ends the induction and then the proof of the proposition.
\end{proof}

\subsection{Injectivity of the fundamental algebraic system}
\label{se:injectivity}
The matrix $A$ appears to be strongly related to the Jacobi curves, and we will use the properties of the latter to deduce the non degeneracy of $A$. Let us denote by $u(\lambda)$ the coordinates of $\lambda \in T^*M$.

\begin{Prop}
\label{th:nonzer.min}
If $\lambda \in T^*M \setminus h_1^{-1}(0)$ is ample with respect to $g_1$, then $A(u(\lambda))$ is injective. As a consequence, there exists at least one $(n-m) \times (n-m)$ minor of the matrix $A(u)$ which is a non identically zero function of $u$.
\end{Prop}	
This proposition results directly from the following lemma combined with Theorem~\ref{th:Agr}.

\begin{Lemma} \label{Theorem.rank}
Let $s$ be a positive integer. Denote by $A_s$ the $sm \times (n-m)$ matrix formed by the first  $s$ layers of $A$. Then
\begin{equation}
\label{rank.A}
\mathrm{rank} A_{s}(u) = \dim J^{(s+1)}_{\lambda}-n-m.
\end{equation}
\end{Lemma}

\begin{proof}
We begin by proving that, for any positive integer $s$,
\begin{equation}
\label{JC.equal.A}
	 (\mathrm{ad} \vec{h})^{s}Y_j  = \sum_{k = m+1}^{n} a^s_{j,k} Y_{k} \mod J^{(s)}_{\lambda}, \quad 1 \leq j \leq m.
	 \end{equation}
Remark first that, for $k=1,\dots,n$,
\begin{eqnarray*}
[\vec{h}_1, Y_k]= \left[ \sum_{i=1}^{m} u_i Y_i + \sum_{i=1}^{m} \sum_{j,l=1}^{n} c^{l}_{ij} u_i u_l \partial_{u_j}, Y_k \right] &=& \sum_{i=1}^{m} u_i [Y_i,Y_k] \mod J_{\lambda}, \\ &=& \sum_{i=1}^{m}u_i \sum_{l=1}^n c^l_{ik} Y_l \mod J_{\lambda},
\end{eqnarray*}
which writes as
\begin{equation}
\label{h.in.coord}
[\vec{h}_1, Y_k]= \sum_{l=1}^{n} q_{kl} Y_l \mod J_{\lambda}.
\end{equation}
Let us prove \eqref{JC.equal.A} by induction on $s$. The case $s=1$ is a direct consequence of \eqref{h.in.coord} since the latter implies that, for $j=1,\dots,m$,
$$
[\vec{h}_1, Y_j]= \sum_{k=m+1}^{n} q_{jk} Y_k +  \sum_{k=1}^{m} q_{jk} Y_k \mod J_{\lambda} = \sum_{k = m+1}^{n} a^1_{j,k} Y_{k} \mod J^{(1)}_{\lambda}.
$$
Assume now that \eqref{JC.equal.A} is satisfied for a given $s$. Using the induction hypothesis, we write
$$
 (\mathrm{ad}\vec{h}_1)^{s+1} Y_j = \left[\vec{h}_1, (\mathrm{ad}\vec{h}_1)^{s}Y_j\right] = \left[\vec{h}_1, \sum_{k = m+1}^{n} a^s_{j,k} Y_{k}  \right] \mod J^{(s+1)}_{\lambda},
 $$
since $[\vec{h}_1, 	 J^{(s)}_{\lambda}] \subset J^{(s+1)}_{\lambda}$.
The last bracket above expands as
\begin{eqnarray*}
 \left[\vec{h}_1, \sum_{k = m+1}^{n} a^s_{j,k} Y_{k}  \right] &=& \sum_{k = m+1}^{n} \vec{h}_1 ( a^s_{j,k}) Y_{k} + \sum_{k = m+1}^{n} a^s_{j,k} \left[\vec{h}_1, Y_{k}  \right], \\
   &=&  \sum_{k = m+1}^{n} \vec{h}_1 ( a^s_{j,k}) Y_{k} + \sum_{k = m+1}^{n} a^s_{j,k} \sum_{l=1}^{n} q_{kl} Y_l \mod J_{\lambda},
\end{eqnarray*}
thanks to \eqref{h.in.coord}. Splitting and renumbering the second sum above, we obtain
\begin{eqnarray*}
 (\mathrm{ad}\vec{h}_1)^{s+1} Y_j &=&  \sum_{k = m+1}^{n} \left( \vec{h}_1 ( a^s_{j,k}) + \sum_{l = m+1}^{n} a^s_{j,l}  q_{lk} \right) Y_k + \sum_{l=1}^{m}\sum_{k = m+1}^{n} a^s_{j,k} q_{kl} Y_l \mod J^{(s+1)}_{\lambda},\\
   &=& \sum_{k = m+1}^{n} a^{s+1}_{j,k} Y_k  \mod J^{(s+1)}_{\lambda},
\end{eqnarray*}
which ends the induction and proves \eqref{JC.equal.A}.

Now,  from Lemma~\ref{le:jacobi}, for any positive integer $s$ there holds $J^{(s+1)}_{\lambda} = J^{(1)}_{\lambda} + \mathrm{span}\{ (\mathrm{ad}\vec{h}_1)^k Y_j (\lambda) \ \mid \ 1 \leq k \leq s, \ 1 \leq j \leq m\}$. Thus it results from~\eqref{JC.equal.A} that
$$
\dim J^{(s+1)}_{\lambda} = \dim J^{(1)}_{\lambda} + \mathrm{rank} A_s(u(\lambda)), \quad \hbox{where } A_s =  \left(
		 \begin{array}{c}
		A^1 \\
		A^2 \\
		 \vdots \\
		 A^{s}
		 \end{array}
		 \right).
$$
Since $\dim J^{(1)}_{\lambda}=n+m$ for any $\lambda$, the lemma is proved.
\end{proof}

A first consequence of the injectivity of $A$ is that the system of equations $A \widetilde{\Phi} = b$ is a sufficient condition for $\Phi$ to be an orbital diffeomorphism.

\begin{Prop}
\label{prop:inverse}
 Consider a local frame $\{ X_1,\dots,X_n\}$ of $D$ on an open subset $U \subset M$, and smooth positive functions $\alpha_1, \dots , \alpha_m$ on $U$. Let $A$ and $b$ be the associated matrices defined by \eqref{elem.A} and \eqref{term.of.b.rec.form}, and denote by $g_1$ and $g_2$ the sub-Riemannian metrics defined locally by the orthonormal frames $X_1,\dots,X_m$ and $\frac{X_1}{\alpha_1},\dots,\frac{X_m}{\alpha_m}$ respectively.

 Assume that $\lambda \in T^*U$ is ample with respect to $g_1$, and that there exists a solution $\widetilde{\Phi}=(\Phi_{m+1}, \dots, \Phi_n)$ to $A \widetilde{\Phi} = b$ near $\lambda$. Let $\Phi_1, \dots, \Phi_m$ be defined by \eqref{eq:first.m.coord}. Then the local smooth fiber-preserving map $\Phi: H_1 \to H_2$ defined by $u_i \circ \Phi=\Phi_i$, $i=1,\dots,n$, satisfies \eqref{orbeq}.
\end{Prop}

\begin{proof}
Following Lemma~\ref{prop}, it is sufficient to prove that $\widetilde{\Phi}$ satisfies \eqref{eq:220} and \eqref{eq:221} near  $\lambda$. The equations of the first layer, i.e.\ $A^1\widetilde{\Phi} = b^1$, are exactly \eqref{eq:220}, hence we are left with the task of proving that $\widetilde{\Phi}$ satisfies \eqref{eq:221}.

Fix a positive integer $s$ and $j \in \{1,\dots,m\}$. Let us write the $j$th row of the system $A^{s} \Phi = b^{s}$,
$$
\sum_{k = m+1}^{n} a^s_{j,k} \Phi_{k} = b^s_j,
$$
 and differentiate this expression  in the direction $\vec{h}_1$. We thus obtain
$$
\sum_{k = m+1}^{n} \vec{h}_1 (a^s_{j,k}) \Phi_{k} + \sum_{k = m+1}^{n} a^s_{j,k} \vec{h}_1 (\Phi_{k}) = \vec{h}_1 (b^s_j).
$$
Write now the $j$th row of the system $A^{s+1} \Phi = b^{s+1}$, replacing the coefficients  by their recurrence formula,
\begin{multline*}
\sum_{k= m+1}^{n} \vec{h}_1(a^s_{j,k})\Phi_{k}  + \sum_{k,l= m+1}^{n} a^s_{j,l} q_{l,k}  \Phi_{k} \\
= \vec{h}_1(b^s_j)  - \frac{1}{\aa} \sum_{k = m+1}^{n} a^s_{j,k} \sum_{i=1}^{m} u_i \left(\alpha_i^2 q_{ki} + \frac{X_k(\alpha_i^2)}{2} u_i\right),
\end{multline*}
and take the difference between the last two formulas. Rearranging the order of summation we obtain
\begin{equation}
\label{eq:Psi}
\sum_{k = m+1}^{n} a^s_{j,k} \left( \vec{h}_1(\Phi_k) - \sum_{l= m+1}^{n} q_{k,l} \Phi_{l} +\frac{1}{\aa} \sum_{i=1}^{m} u_i \left(\alpha_i^2 q_{ki} + \frac{X_k(\alpha_i^2)}{2} u_i\right) \right) = 0.
\end{equation}

Denote by $\Psi_{k}$ the terms inside the bracket above, and set $\Psi=(\Psi_{m+1}, \dots, \Psi_{n})$. Formula \eqref{eq:221} for $\widetilde{\Phi}$ is exactly $\Psi=0$. From \eqref{eq:Psi}, the vector $\Psi$ satisfies the system $A \Psi = 0$. Moreover, by Proposition~\ref{th:nonzer.min} the matrix $A(u)$ has full rank at $u=u(\lambda)$, and hence in a neighborhood of $u(\lambda)$ in $T^*M$. On this neighborhood $\Psi$ must be identically zero, which implies that $\widetilde{\Phi}$ satisfies \eqref{eq:221}. The statement is proved.
\end{proof}


\section{First divisibility and consequences}
\label{se:integrability}

\subsection{First divisibility}
In \cite{Z}, Zelenko introduced an algebraic condition called \emph{first divisibility condition}, which implies interesting conditions on the eigenvalues $\alpha_i^2$ and on the structure coefficients.

Consider two sub-Riemannian metric $g_1,g_2$ on $(M,D)$,  a stable point $q_0$ with respect to these metrics, and introduce as in Section~\ref{se:coord_orbdiffeo} a frame $\{X_1,\dots,X_n\}$ adapted to $(g_1,g_2)$ and the associated coordinates $(u_1,\dots,u_n)$ on the fibers of $T^*M$.

Set $\mathcal{P}= \alpha_1^2 u_1^2 + \cdots + \alpha_m^2 u_m^2$, where $\alpha_1^2, \dots,\alpha_m^2$ are the eigenvalues of the transition operator. Note that $\mathcal{P}$ and its Lie-derivative $\vec{h}_1(\mathcal{P})$ are polynomial functions on the fiber, i.e.\ polynomial functions of $u$ (see \cite[Eq.~(2.30)]{Z} for an intrinsic definition of $\mathcal{P}$). We say that the ordered pair of sub-Riemannian metrics $(g_1,g_2)$ satisfies \emph{the first divisibility condition} if the polynomial $\vec{h}_1(\mathcal{P})$ is divisible by $\mathcal{P}$.

\begin{Prop}[\cite{Z}, Proposition 6]
\label{prop:from.div}
	If $(g_1, g_2)$ and $(g_2,g_1)$ satisfy the first divisibility condition in a neighborhood $U$ of a stable point $q_0$, then for any $q \in U$ the following properties hold:
	 \begin{itemize}
	\item $\text{for any } 1 \leq i,j \leq m$, $\quad [X_i,X_j](q) \notin D(q) \ \Rightarrow \ \alpha_i(q) = \alpha_j(q)$;\smallskip
\item $X_i\left(\frac{\alpha_j^2}{\alpha_i^2}\right) = 2 c^j_{ij} \left(1 - \frac{\alpha_j^2}{\alpha_i^2}\right)\quad \text{for any } \ 0 \leq i,j \leq m$; \smallskip
		 \item $X_i\left(\frac{\alpha_j^2}{\alpha_i}\right) = 0, \quad \alpha_j \neq \alpha_j$;\smallskip
		 \item $X_i\left(\frac{\alpha_j}{\alpha_k}\right) = 0, \quad \alpha_i \neq \alpha_j,\ \alpha_i \neq \alpha_k$.
	 \end{itemize}
\end{Prop}

It appears actually that this condition is always fulfilled by pairs of metrics whose Hamiltonian vector fields are orbitally diffeomorphic.

\begin{Prop}
\label{prop:divty}
If $\vec{h}_1, \vec{h}_2$ are orbitally diffeomorphic near some $\lambda \in \pi^{-1}(q_0)$, then $(g_1,g_2)$ and $(g_2,g_1)$ satisfy the first divisibility condition near $q_0$.
\end{Prop}

\begin{proof}
Let $\Phi$ be the orbital diffeomorphism between the extremal flows of $g_1,g_2$. From Proposition~\ref{Aprop}, the $n-m$ last coordinates of $\Phi$ satisfy $A \widetilde{\Phi} = b$. Let us give first some algebraic properties of the components $\Phi_i$.

Notice that
$$
\aa^2 = \frac{\mathcal{P}}{h_1},
$$
which implies that
\begin{equation}
\label{eq:h1P}
\frac{\vec{h}_1(\mathcal{P})}{\mathcal{P}} = \frac{\vec{h}_1(\aa^2)}{\aa^2}.
\end{equation}
Using this remark, a simple induction argument shows that, for any positive integer $s$, there exists a constant $C_s>0$ and polynomial functions $\mathrm{pol}_{s,j}(u)$ on the fiber such that
\begin{equation}
\label{eq:alphab}
 b_j^s = \frac{C_s \alpha_j^2 u_j}{\aa} \left( \frac{\vec{h}_1(\mathcal{P})}{\mathcal{P}} \right)^s + \frac{1}{\aa \mathcal{P}^{s-1}} \mathrm{pol}_{s,j}(u), \qquad j=1,\dots,m.
\end{equation}
From Proposition~\ref{th:nonzer.min}, the matrix $A$ admits at least one nonzero maximal minor $\delta$. Since all coefficients of $A$ are polynomial functions of $u$, $\delta$ is in turn polynomial in $u$. Using Cramer's rule, we deduce from \eqref{eq:alphab} that there exists an integer $S$ such that, for $i=m+1,\dots,n$,
\begin{equation}
\label{eq:Phipol}
\Phi_i = \frac{1}{\aa \delta \mathcal{P}^{S}} \times \hbox{polynomial in } u.
\end{equation}

Let us prove now the divisibility of  $\vec{h}_1(\mathcal{P})$ by $\mathcal{P}$. Choose an arbitrary large integer $s$ ($s>S$) and $j \in \{1,\dots,m\}$, and consider the $j$th equation of the $s$th layer of the system \eqref{A.phi.B},
$$
	a^s_{j, m+1} \Phi_{m+1} + \dots + a^s_{j, n-m} \Phi_{n-m} = b^s_j.
$$
Recall that all coefficients $a^s_{i,j}$ are polynomial functions of $u$. Substituting expressions \eqref{eq:alphab} and \eqref{eq:Phipol} for $b^s_j$ and  the $\Phi_i$'s respectively, we get,
$$
\frac{C_s \alpha_j^2 u_j}{\aa} \left( \frac{\vec{h}_1(\mathcal{P})}{\mathcal{P}} \right)^s + \frac{1}{\aa \mathcal{P}^{s-1}} \mathrm{pol}_{s,j}(u) = \frac{1}{\aa \delta \mathcal{P}^{S}} \times \hbox{polynomial in } u.
$$
Multiplying by $\aa \mathcal{P}^{s}$, we obtain finally,
\begin{equation}
\label{eq:irreduc}
C_s \alpha_j^2 u_j \vec{h}_1(\mathcal{P})^s = \mathcal{P} \mathrm{pol}_{s,j}(u) + \frac{\mathcal{P}^{s-S}}{\delta} \times \hbox{polynomial in } u.
\end{equation}

Assume by contradiction that $\vec{h}_1(\mathcal{P})$ is not divisible by $\mathcal{P}$.
Let $k$ be the maximal nonnegative integer such that $\delta$ is divisible by $\mathcal{P}^k$. Set $\delta=\mathcal{P}^k \delta_1$ and take $s=k+S+1$. Then \eqref{eq:irreduc} writes as
\begin{equation}
\label{eq:irreduc'}
C_s \alpha_j^2 u_j \vec{h}_1(\mathcal{P})^s = \mathcal{P} \mathrm{pol}_{s,j}(u) + \mathcal{P}\frac{\mathrm{pol}(u)}{\delta_1},
\end{equation}
which implies that $\mathcal{P}\mathrm{pol}(u)/\delta_1$ is polynomial in $u$. Taking into account that $\mathcal{P}$ is a positive quadratic form, and thus it is irreducible over $\mathbb R$, we obtain that $\mathrm{pol}(u)/\delta_1$ is polynomial. Therefore the
 right-hand side of \eqref{eq:irreduc'} is divisible by $\mathcal{P}$ and $\vec{h}_1(\mathcal{P})^s$ is also divisible by $\mathcal{P}$. We have a contradiction, which completes the proof.

%
\end{proof}

This proposition has several consequences. The first one is an obvious  corollary of Propositions \ref{prop:from.div} and~\ref{prop:divty}. Let us introduce first some notations. Let $N=N(q_0)$ be the number of distinct eigenvalues of the transition operator $S_q$ for $q$ near $q_0$. We assume that the eigenvalues $\alpha_i^2$, $i=1,\dots,m$, are numbered in such a way that $\alpha_1^2, \ldots, \alpha_N^2$ are the $N$ distinct ones.  For $\ell=1, \dots, N$, we denote by $I_\ell$ the set of indices $i \in \{1,\dots,m\}$ such that $\alpha_i=\alpha_\ell$.

\begin{Coro}
\label{coro:from.div2}
Assume $\vec{h}_1, \vec{h}_2$ are orbitally diffeomorphic near some $\lambda \in \pi^{-1}(q_0)$. Then, for any $\ell, \ell' \in \{1, \dots, N\}$, $\ell \neq \ell'$,
	 \begin{enumerate}[(i)]
		 \item $[X_\ell,X_{\ell'}] \in D^1$;
        \item if $X \in \mathrm{Lie}\{X_i \ : \ i \in I_\ell\}$, then $X \left(\frac{\alpha_{\ell'}^2}{\alpha_{\ell}}\right) = 0$;
		 \item if $X \in \mathrm{Lie}\{X_i \ : \ i \in I_\ell\}$ and $\ell'' \neq \ell$, then $X\left(\frac{\alpha_{\ell'}}{\alpha_{\ell''}}\right) = 0$.
	 \end{enumerate}
\end{Coro}

The second consequence results directly from the definition of first-divisibility.

\begin{Lemma}
\label{le:PQ}
If $\vec{h}_1, \vec{h}_2$ are locally orbitally diffeomorphic, then
\begin{equation}
	 \label{Q}
	\vec{h}_1 (\mathcal P)=Q \mathcal P, \ \hbox{ where } \
	 Q=\sum_{i=1}^m \frac{X_i(\alpha_i^2)}{\alpha_i^2} u_i.
\end{equation}
\end{Lemma}

\begin{proof}
From Proposition~\ref{prop:divty}, the third degree polynomial $\vec{h}_1 (\mathcal{P})$ is divisible by the quadratic polynomial $\mathcal{P}$. Hence there exists a linear function $Q=\sum_{j=1}^n p_j u_j$ such that
$$
\vec{h}_1 (\mathcal{P})=Q \mathcal{P}= \bigl(\sum_{j=1}^n p_j u_j\bigr)\bigl(\sum_{i=1}^m\alpha_{i}^2 u_i^2\bigr).
$$
On the other hand, using the expression \eqref{eq:h1} of $\vec{h}_1$, we get
$$
\vec{h}_1(\mathcal{P})= \sum_{i,j=1}^{m} X_i(\alpha_j^2) u_i u_j^2 + \sum_{i,j=1}^{m} \sum_{k=1}^{n} 2 c^{k}_{ij} \alpha_j^2 u_i u_j u_k.
$$
Identifying the coefficients of the monomials $u_i^3$ and $u_i^2u_j$, $1\leq i \leq m$, $m<j \leq n$, in the two expressions above, we obtain respectively
$$
p_i= \frac{X_i(\alpha_i^2)}{\alpha_i^2} \hbox{ for } 1\leq i \leq m, \quad p_i=0 \hbox{ for } m< i \leq n.
$$
\end{proof}

\subsection{Existence of first integrals}

An important consequence of the first-divisibility property is the existence of quadratic first integrals for the Hamiltonian flow.
Let $g_1,g_2$ be two sub-Riemannian metrics on $(M,D)$, and $q_0$ be a stable point w.r.t.\ $g_1,g_2$. Proceeding as above, we assume that the eigenvalues $\alpha_i^2$, $i=1,\dots,m$, are numbered in such a way that $\alpha_1^2, \ldots, \alpha_N^2$ are the $N$ distinct ones. We introduce also  a frame $\{X_1,\dots,X_n\}$ adapted to $(g_1,g_2)$, the associated coordinates $(u_1,\dots,u_n)$ on the fibers of $T^*M$, and the polynomial
$$
\mathcal{P}=\sum_{i=1}^m\alpha_{i}^2 u_i^2.
$$

\begin{Prop}
	 \label{1divntegral}
If $\vec{h}_1$ and $\vec{h}_2$ are orbitally diffeomorphic near some $\lambda \in \pi^{-1}(q_0)$, then the function
\begin{equation*}
\label{integral}
F= \Bigl(\prod_{\ell=1}^N\alpha_\ell^2\Bigr)^{-\frac{2}{N+1}} \mathcal{P}
\end{equation*}
is a  first integral of the normal extremal flow of $g_1$, i.e.
\begin{equation*}
	 \label{integraldef}
	\vec h_1(F)=0.
\end{equation*}
\end{Prop}

Note that, in the Riemannian case (i.e.\ $D=TM$), the existence of this quadratic first integral was shown by Levi-Civita in \cite{Levi-Civita1896} (see also \cite{Matveev-Topalov2003}, where this integral is attributed to Painlev\'{e}).

\begin{proof}
Set $f=\Bigl(\prod_{\ell=1}^N\alpha_\ell^2\Bigr)^{-\frac{2}{N+1}}$. Using Lemma~\ref{le:PQ} we get
	 \begin{equation}
	 \label{Q2}
	\vec h_1(F)= \vec h_1(f \mathcal{P})= \bigl(\vec h_1(f)+f Q \bigr) \mathcal{P}.
	 \end{equation}
Further, using the expression \eqref{eq:h1} of $\vec h_1$, we have
\begin{align}
\vec h_1(f) &=-\frac{2}{N+1}\sum_{i=1}^m\left(\Bigl(\prod_{\ell=1}^N\alpha_\ell^2\Bigr)^{-\frac{2}{N+1}-1}u_i\sum _{\ell=1}^N
	 \bigl(\prod_{k\neq \ell}\alpha_k^2\bigr) X_i(\alpha_\ell^2)\right) \nonumber \\
	& = -\frac{2}{N+1}f\sum_{i=1}^m u_i\sum _{\ell=1}^N
	\cfrac {X_i(\alpha_\ell^2)}{\alpha_\ell^2}.\label{Q3}
\end{align}

Notice now that Corollary~\ref{coro:from.div2}, \emph{(ii)}, implies that,
\begin{equation*}
\label{Q4}
\text{if} \ \alpha_j \neq \alpha_i, \quad \text{ then} \quad X_i(\alpha_j^2)=\cfrac{\alpha_j^2 X_i(\alpha_i^2)}{2\alpha_i^2}.
\end{equation*}
Plugging this into \eqref{Q3}, we get
\begin{equation*}
\vec{h}_1(f)
= -	 \frac{2}{N+1}\Bigl(\cfrac{N-1}{2}+1\Bigr)f\sum _{i=1}^m\frac{X_i(\alpha_{i}^2)}{\alpha_{i}^2} u_i=
	-fQ.
\end{equation*}
By \eqref{Q2} we obtain $\vec h_1(F)=0$, which completes the proof.
\end{proof}

The normal extremal flow of $g_1$ already admits $h_1$ as a quadratic first integral, and $F$ is not proportional to $h_1$ except when $N=1$, which corresponds to the case where $g_1$ and $g_2$ are conformal to each other. This proves Theorem \ref{th:integral}. The existence of several quadratic first integrals appears to be a strong condition on the metric.

\begin{Prop}
\label{le:quadratic}
Let $(M,D)$ be fixed. The normal extremal flow of a generic sub-Riemannian metric on $(M,D)$ admit no other non-trivial quadratic first integral than its Hamiltonian.
\end{Prop}
One can find a rigorous proof of this result in \cite{Kruglikov-Matveev2015} for the case $D=TM$ (Riemannian case), and we show in Appendix~\ref{se:proof_quadratic} that the same arguments can be applied to any bracket-generating distribution on $M$.

Theorem~\ref{th:generic_metric} is a direct consequence of this proposition, Proposition~\ref{1divntegral}, and Corollary~\ref{le:conf_affine} below.

\subsection{Consequences on affine equivalence}
\begin{Prop}
\label{le:constant_alphai}
If two sub-Riemannian metrics $g_1, g_2$ on $(M, D)$ are affinely equivalent on an open connected subset $U \subset M$, then all the eigenvalues $\alpha_1^2, \dots,\alpha_m^2$ of the transition operator are constant.
\end{Prop}

\begin{proof}
Let $g_1,g_2$ be two affinely equivalent metrics on $U$, and let $q_0 \in U$ be a stable point with respect to $g_1,g_2$. From Proposition~\ref{th:orb.to.proj}, $\vec{h}_1$ and $\vec{h}_2$ are locally orbitally diffeomorphic and $\vec{h}_1(\aa)=0$. Using equality \eqref{eq:h1P} we get that $\vec{h}_1(\mathcal{P})=0$.

From Lemma~\ref{le:PQ}, $\vec{h}_1 (\mathcal P)=Q \mathcal P$. Hence $Q=0$,  which implies that $X_i(\alpha_i^2)=0$ for $i=1,\dots,m$. Using Corollary~\ref{coro:from.div2} \emph{(ii)}, we obtain $X_i(\alpha_j^2)$ for any $i,j \in \{1,\dots,m\}$, and since the vector fields $X_1,\dots,X_m$ are bracket-generating, we finally get that $\alpha_1^2,\dots,\alpha_m^2$ are constant near $q_0$.

Thus the eigenvalues $\alpha_1^2,\dots,\alpha_m^2$ are continuous functions on $U$ which are locally constant near stable points. Since the set of stable points is dense in $U$, we conclude that all eigenvalues are constant.
\end{proof}

\begin{rk}
\label{re:affine}
It results from \eqref{Q} that, if all $\alpha_i$'s are constants, then $\vec{h}_1 (\mathcal P)=0$, which in turn implies $\vec{h}_1 (\aa)=0$ by \eqref{eq:h1P}. Combining this remark with Propositions \ref{th:orb.to.proj} and~\ref{th:proj.to.orb}, we obtain the following result: if two metrics are projectively equivalent and the transition operator has constant eigenvalues, then the metrics are affinely equivalent.
\end{rk}

\begin{Coro}
\label{le:conf_affine}
If a sub-Riemannian metric is conformally projectively rigid, then it is affinely rigid.
\end{Coro}

\begin{proof}
Let $g_1$ be a conformally projectively rigid sub-Riemannian metric. If a metric $g_2$ is affinely equivalent to $g_1$, then it is also projectively equivalent to $g_1$, and by hypothesis $g_2 = \alpha^2 g_1$. Hence $\alpha^2$ is the unique eigenvalue of the transition operator and is constant by Proposition~\ref{le:constant_alphai}, which implies that $g_2$ is trivially equivalent to $g_1$.
\end{proof}



\section{Levi-Civita pairs}
\subsection{Definition and the main open question}
\label{se:LCpairs}

Let us introduce a special case of non-trivially projectively and affinely equivalent metrics.  First we define a distribution which admits a product structure.

Fix positive integers $N$, $n_1,\ldots,n_N$, and set $n=n_1+\cdots+n_N$. We denote the canonical coordinates on $\R^n= \R^{n_1}\times \cdots \times \R^{n_N}$ by $x=(\bar x_1,\ldots,\bar x_N)$, where $\bar x_\ell=(x_\ell^1,\ldots, x_\ell^{n_\ell})$. For any $\ell \in \{1,\dots,N\}$, let $D_\ell$ be a Lie bracket generating distribution on $\mathbb R^{n_\ell}$. We define the \emph{product distribution} $D=D_1 \times \cdots \times D_N$ on $\mathbb R^n$ by
\begin{equation}
\label{eq:product_distrib}
D(x)=\left\{ v \in T_{x} \mathbb R^n: (\pi_\ell)_*(v)\in D_\ell \bigl(\pi_\ell(x)\bigr), \ \ell=1,\ldots, N \right\},
\end{equation}
where  $\pi_\ell: \R^n \rightarrow \R^{n_\ell}$, $\ell=1,\ldots, N$, are the canonical projection.

\begin{Def}
\label{def:product}
We say that a distribution $D$ on a $n$-dimensional manifold $M$ admits \emph{a nontrivial product structure} at $q \in M$ if there is a local coordinate system in a neighborhood of $q$ in which $D$ takes the form of a product distribution with $N\geq 2$ factors.
\end{Def}
Note that the case $N=1$ is trivial since any distribution can be written in local coordinates as a product distribution with one factor.

\begin{rk}
Regarding the notion of distribution admitting a product structure one can have in mind the following different notion, which is weaker than the one we use: a distribution $D$ admits a \emph{weak product structure} if there are two sub-distributions $D_1$ and $D_2$ of $D$ satisfying the following two properties:
\begin{enumerate}
\item	
$D(q)=D_1(q)\oplus D_2(q)$ for any $q\in M$,
\item  there are local frames $(X_1,\ldots, X_{m_1})$ of $D_1$ and $(Y_1,\ldots, Y_{m_2})$ of $D_2$ such that $[X_i, Y_j]=0$,
\end{enumerate}
and in this case $D$ is said to be a weak product of $D_1$ and $D_2$.
We stress that in our Definition~\ref{def:product} we require much more: if $D=D_1\times D_2$ in our sense, then $D$  is clearly a weak product of $D_1$ and $D_2$. Moreover if, for some $j>1$, $D^j$, $D_1^j$, and $D_2^j$  are still distributions, then $D^j$ must be a weak product of $D_1^j$ and $D_2^j$. For example, a contact distribution $D$ does not admit a product structure in our sense, but it does admit a weak product structure if $\mathrm{rank} \,D\geq 4$.
\end{rk}

Given a product distribution $D=D_1 \times \cdots \times D_N$ on $\mathbb R^n$, we choose for every $\ell \in \{1,\dots,N\}$ a sub-Riemannian metric  $\bar g_\ell$ on  $(\mathbb R^{n_\ell},D_\ell)$ and a function $\beta_\ell$ depending only on the variables $\bar x_\ell$ such that $\beta_\ell$ is constant if $n_\ell>1$ and  $\beta_\ell (0)\neq \beta_{\ell'}(0)$ for ${\ell}\neq \ell'$. We define two sub-Riemannian metrics $g_1,g_2$ on $(\R^n,D)$ by
\begin{equation}
\label{met1}
\left\{
\begin{array}{l}
\displaystyle  g_1 (x)(\dot x, \dot x) =\sum_{\ell=1}^N \gamma_\ell (x) \: \bar g_\ell(\bar{x}_\ell)(\dot{\bar{x}}_\ell, \dot{\bar{x}}_\ell), \\
\displaystyle  g_2 (x)(\dot x, \dot x) =\sum_{\ell=1}^N \alpha^2_\ell(x) \gamma_\ell(x) \: \bar g_\ell(\bar{x}_\ell)(\dot{\bar{x}}_\ell, \dot{\bar{x}}_\ell),
\end{array}
\right.
\end{equation}
where
\begin{equation}
\label{met2}
\alpha^2_\ell(x)=\beta_\ell(\bar x_\ell)\prod_{{\ell'}=1}^N\beta_{\ell'}(\bar x_{\ell'}), \qquad   \gamma_\ell(x)=\prod_{{\ell'}\neq
	\ell}\Bigl|\frac{1}{\beta_{\ell'}(\bar x_{\ell'})}-\frac{1}{\beta_\ell(\bar x_\ell)}\Bigr|.
\end{equation}

\begin{Def}
Let $D$ be a distribution on an $n$-dimensional manifold $M$. We say that a pair $(g_1,g_2)$ of sub-Riemannian metrics on $(M,D)$ form a \emph{(generalized) Levi-Civita pair} at a point $q\in M$, if there is a local coordinate system in a neighborhood of $q$, in which $D$ takes the form of a product distribution and the metrics $g_1$ and $g_2$ have the form \eqref{met1}. We say that such a pair has \emph{constant coefficients} if the coordinate system can be chosen so that the functions $\beta_\ell$, $\ell=1,\ldots, N$, are constant  (and so all functions $\alpha^2_\ell$ and $\gamma_\ell$ are constant too).
\end{Def}

This definition is inspired by the classification in the Riemannian case appearing in \cite{Levi-Civita1896}. Note however that, in the Riemannian case, the distribution $D=TM$ takes the form of a product in any system of coordinates, so that Levi-Civita pairs always exist locally.

\begin{rk}
A Levi-Civita pair with $N=1$ is a pair of conformal metrics, $g_2 = \alpha_1^2 g_1$. If moreover $n>1$, two such metrics are actually constantly proportional. Thus, when $n>1$,  the metrics of a Levi-Civita pair are constantly proportional if and only if $N=1$.
\end{rk}

\begin{Prop}
\label{le:LCpairs}
The two metrics of a Levi-Civita pair are projectively equivalent. They are affinely equivalent if and only if the pair has constant coefficients.
\end{Prop}

The proof of this proposition requires some tedious computations and has been postponed to Appendix~\ref{se:proof_LCpairs}. We can however give here a short proof of the second statement. Indeed, note that in a Levi-Civita pair with constant coefficients, the metrics are actually of the form of product metrics, i.e.\ each of them is a linear combination of metrics $\bar{g}_\ell$ and each of the sub-Riemannian manifolds $(M, D, g_1)$ and $(M, D, g_2)$ is locally a product of some sub-Riemannian manifolds $(M_\ell, D_\ell,\bar{g}_\ell)$. Assume for simplicity that $N=2$ in~\eqref{met1} (the general case can be treated in the same way). Then a trajectory $x(\cdot) = (\bar{x}_1, \bar{x}_2)(\cdot)$ is an energy minimizer of $g_1$ if and only if $\bar{x}_1(\cdot)$ and $\bar{x}_2(\cdot)$ are energy minimizers of $\bar{g}_1$ and $\bar{g}_2$ respectively. The same holds for $g_2$. As a consequence, the metrics $g_1$ and $g_2$ are affinely equivalent (and not proportional if $\alpha_1^2 \neq \alpha_2^2$).

The main open question is \emph{whether under some natural regularity assumption the generalized Levi-Civita pairs are the only pairs of the projectively equivalent metrics}.

\subsection{Levi-Civita theorem  and  its generalizations}

The preceding question has a positive answer in the Riemannian case, that is when $D=TM$. Indeed, in that case the local classification of projectively equivalent metrics near generic points has been established by \cite{Levi-Civita1896} in any dimension. The classification of affinely equivalent metrics is a consequence of \cite[Th.\ p.~303]{Eisenhart1923}. We summarize all these results in the following theorem.

\begin{Theorem}
\label{LCthm}
Assume $\dim M>1$. Then two Riemannian metrics on $M$ are non-trivially projectively equivalent in a neighborhood of a stable point $q$  if and only if they form a Levi-Civita pair at $q$. They are moreover affinely equivalent if and only if the pair has constant coefficients.
\end{Theorem}

We can actually give a rather short explanation of the classification of affinely equivalent Riemannian metric based on the de Rham  decomposition theorem, \cite{deRham1952}. Indeed, a simple analysis of the geodesic equation implies that two Riemannian  metrics are affinely equivalent if and only if they have the same Levi-Civita connection. Since the Levi-Civita connection is parallel with respect to the metric,  a metric with given Levi-Civita connection on a connected manifold is determined by its value at one point $q$. Besides, it must be invariant with respect to the holonomy group (or the reduced holonomy group for the local version of the problem). If the action of the holonomy group on $T_q M$ is irreducible, the Riemannian metric is uniquely determined by its Levi-Civita connection, i.e.\ it is affinely rigid. On the other hand, if the action of the holonomy group is reducible, then by the de Rham decomposition theorem the Riemannian metric becomes the direct product of Riemannian metrics and any metric which is affinely equivalent to it is such that the metrics can be represented as in \eqref{met1} with all functions $\beta_\ell$ being constant.

Our  main open question has a positive answer as well for sub-Riemannian metrics on contact and quasi-contact distributions, which are typical cases of corank $1$ distributions (i.e.\ $m=n-1$). Recall that a \emph{contact distribution} $D$ on a $(2k+1)$-dimensional manifold $M$, $k>0$,  is a rank-$2k$ distribution for which there exists a 1-form $\omega$ such that  at every $q\in M$, $D(q) = \ker \omega(q)$ and $\left. d \omega(q)\right|_{D(q)}$ is non-degenerate. A \emph{quasi-contact distribution} $D$ on a $2k$-dimensional manifold $M$, $k>1$,  is a rank-$(2k-1)$ distribution for which there exists a  1-form $\omega$ such that at every $q\in M$, $D(q) = \ker \omega(q)$ and $\left. d \omega(q)\right|_{D(q)}$  has a one-dimensional kernel. The main result of \cite{Z} can be formulated in the following way.

\begin{Theorem}[\cite{Z}]
\label{Zelenkothm}
Two sub-Riemannian metrics on a contact or a quasi-contact distribution are non-trivially projectively equivalent at a stable point $q$ if and only if they form a Levi-Civita pair at $q$.
\end{Theorem}

\begin{rk}
This theorem and Proposition~\ref{le:constant_alphai} imply that, in the contact and quasi-contact cases, two affinely equivalent metrics form a Levi-Civita pair with constant coefficients.
\end{rk}

Since contact distributions are never locally equivalent to a non-trivial product distribution, they admit only Levi-Civita pairs with $N = 1$.
\begin{Coro}
On a contact distribution, every sub-Riemannian metric is projectively rigid.
\end{Coro}
For a generic corank one distribution $D$ on an odd dimensional manifold $M$, there is an open and dense subset of $M$ where $D$ is locally contact. By continuity we obtain the following result.
\begin{Coro}
Let $M$ be an odd-dimensional manifold. Then, for a generic corank one distribution on $M$, all metrics are projectively rigid.
\end{Coro}

\section{Left-invariant metrics on Carnot groups}

Let us study the particular case of affine and projective equivalence of left-invariant sub-Riemannian metrics on Carnot groups. This case plays an important role in sub-Riemannian geometry since Carnot groups appear as tangent cones  to sub-Riemannian manifolds near generic points.

\begin{Def}
	\label{def:Carnot}
	A Carnot group $\mathbb{G}$ of step $r \geq 1$ is a connected and simply connected nilpotent Lie group whose Lie algebra $\mathfrak{g}$ admits a step $r$ grading
$$
\mathfrak{g} = V^1 \oplus \cdots \oplus V^r,
$$
and is generated by its first component, that is, $[V^j, V^1] = V^{j+1}$ for $1 \leq j \leq r-1$.  A graded Lie algebra satisfying the last property is called \emph{fundamental}.
\end{Def}

A Carnot group is canonically endowed with a bracket generating distribution $D_\mathbb{G}$: identifying $\mathfrak{g}$ with the tangent space $T_e\mathbb{G}$ to $\mathbb{G}$ at the identity $e$, $D_\mathbb{G}$ is the distribution spanned by the left-invariant vector fields whose value at the identity belongs to $V^1$. Hence $D_\mathbb{G}^k(e) =  V^1 \oplus \cdots \oplus V^k$ for $k \leq r$, and the step $r$ of the Carnot group is exactly the step of the distribution.

Given an inner product on $V^1$, we can extend it to a Riemannian metric on $D_\mathbb{G}$ by left-translations. Such a sub-Riemannian metric on $(\mathbb{G},D_\mathbb{G})$ is called a \emph{left-invariant sub-Riemannian metric on $\mathbb{G}$}.

\begin{Theorem}
\label{th:Carnot.prod}
Let $g_1, g_2$ be two left-invariant sub-Riemannian metrics on a Carnot group $\mathbb{G}$. If $g_1$ and $g_2$ are non-trivially projectively equivalent, then $D_\mathbb{G}$ admits a non-trivial product structure and $(g_1, g_2)$ is a Levi-Civita pair with constant coefficients.
\end{Theorem}

\begin{proof}
Let $g_1, g_2$ be two left-invariant sub-Riemannian metrics on  $\mathbb{G}$ which are non-trivially projectively equivalent. Set $D=D_\mathbb{G}$.
Since both metrics $g_1$ and $g_2$ are obtained by left-invariant extensions of inner products on $V^1$, it is clear that the eigenvalues $\alpha_1^2, \dots, \alpha_m^2$ of the transition operator are constant. Thus the number $N$ of distinct eigenvalues is constant and every point of $\mathbb{G}$ is stable. Note that $N$ is necessarily greater than one, otherwise the metrics would be proportional, i.e.\ trivially equivalent.

We choose the numbering of the eigenvalues $\alpha_i^2$, $i=1,\dots,m$, in such a way that $\alpha_1^2, \ldots, \alpha_N^2$ are the $N$ distinct ones. Let $X_1, \dots, X_m$ be a $g_1$-orthonormal basis of $V^1$ such that $\frac{1}{\alpha_1}X_1, \dots, \frac{1}{\alpha_m}X_m$ is orthonormal with respect to $g_2$.
For $\ell=1, \dots, N$, we denote by $I_\ell$ the set of indices $i \in \{1,\dots,m\}$ such that $\alpha_i=\alpha_\ell$, and by $V^1_\ell$ the linear subspace of $V^1$ generated by the vectors $X_i$, $i \in I_\ell$. We get
\begin{equation}
\label{eq:V1_direct_sum}
V^1=V^1_1 \oplus \cdots \oplus V^1_N.
\end{equation}
Each subspace $V^1_\ell$, $\ell=1, \dots, N$, generates a graded Lie subalgebra of $\mathfrak{g}$,
$$
\mathfrak{g}_\ell = V^1_\ell \oplus \cdots \oplus V^r_\ell, \quad \hbox{where } V^{k+1}_\ell=[V^k_\ell, V^1_\ell].
$$
Moreover, from Corollary~\ref{coro:from.div2} \emph{(i)}, we have
$$
[V^1_\ell, V^1_{\ell'}] = 0 \quad \hbox{for all } \ell \neq \ell' \in \{1, \dots, N\}.
$$
Using the Jacobi identity, this relation can be generalized as
\begin{equation}
\label{eq:zero bracket}
[V^k_\ell, V^s_{\ell'}] = 0 \quad \hbox{for } \ell \neq \ell' \in \{1, \dots, N\}, \ k,s \in \{1, \dots, r\}.
\end{equation}
Hence each homogeneous component $V^k$, $k=1, \dots, r$, admits a decomposition into a sum $V^k=V^k_1 + \cdots + V^k_N$, and the Lie algebra $\mathfrak{g}$ writes as
\begin{equation}
\label{eq:directsum}
\mathfrak{g} = \mathfrak{g}_1 + \cdots + \mathfrak{g}_N.
\end{equation}
Note that, if the sum \eqref{eq:directsum} is a direct one, then \eqref{eq:V1_direct_sum} implies that the distribution $D$ admits a product structure $D=D_1 \times \cdots \times D_N$, where $D_\ell$, $\ell=1,\dots, N$,  is the distribution spanned by the left-invariant vector fields whose value at the identity belongs to $V^1_\ell$. Thus, in order to prove that $D$ admits a non-trivial product structure, it is sufficient to prove that the sum \eqref{eq:directsum} is a direct sum, i.e.\ $\mathfrak{g}_\ell \cap \mathfrak{g}_{\ell'}$ is reduced to zero when $\ell \neq \ell'$.

The first step is to complete $\{X_1,\dots, X_m\}$ into a basis of $\mathfrak{g}$ adapted to the grading $\mathfrak{g} = V^1 \oplus \cdots \oplus V^r$. For $k=2,\dots,r$, we construct a basis of $V^k=V^k_1 + \cdots + V^k_N$ as follows.
Fix first a basis of $\cap_{1\leq  \ell\leq N} V^k_\ell$; then complete it into a basis of $\mathrm{span}\{\cup_{1\leq i\leq N} \left(\cap_{\ell \neq i} V^k_\ell \right)\}$ with vectors from $\cup_{1\leq i\leq N} \left(\cap_{\ell \neq i} V^k_\ell\right)$; then to a basis of  $\mathrm{span}\{\cup_{1\leq i\leq N} \left(\cup_{1\leq j\leq N} \left(\cap_{\ell\neq i, \ell \neq j} V^k_\ell\right)\right)\}$ with vectors from $\cup_{1\leq i\leq N} \left(\cup_{1\leq j\leq N} \left(\cap_{\ell\neq i, \ell\neq j} V^k_\ell \right)\right)$ and so on. At the last step, complete the obtained set of vectors into a basis of $V^k$.

By collecting the basis of $V^1, V^2,\dots , V^r$, we obtain a basis $\{X_1,\dots, X_n\}$ of $\mathfrak{g}$. By abuse of notations, we keep the notation $X_i$ to denote the left-invariant vector field whose value at identity is $X_i$. We have constructed in this way a frame $\{X_1,\dots, X_n\}$ of $T\mathbb{G}$ with the following properties:
\begin{itemize}
  \item it is by construction a frame adapted to $(g_1,g_2)$;
  \item it contains a basis of every $D^k_\ell$, $k=1,\dots, r$, $\ell=1,\dots,N$; for $\ell=1,\dots,N$, we denote by $\mathcal{L}(I_\ell)$ the set of indices such that $\{X_i, \ i \in \mathcal{L}(I_\ell)\}$ is a basis of $D^r_\ell$;
  \item from \eqref{eq:zero bracket}, $[X_i,X_j]=0$ if $i$ and $j$ belong to two different sets $\mathcal{L}(I_\ell)$; this implies the following property of the structure coefficients:
\begin{equation}
\label{eq:struct_coeff}
\hbox{if $i,j,k$ do not belong to the same $\mathcal{L}(I_\ell)$, then } c_{ij}^k=0;
\end{equation}
\item all structure coefficients are constant since the vector fields are left-invariant; moreover,
\begin{equation}
\label{eq:nilp}
\hbox{if $w_k \neq w_i+w_j$, then } c_{ij}^k=0,
\end{equation}
where as usual $w_i$ is the smallest integer $l$ such that $X_i \in D^l$.
\end{itemize}
The property $\mathfrak{g}_\ell \cap \mathfrak{g}_{\ell'} = \{0\}$ is equivalent to $\mathcal{L}(I_\ell) \cap \mathcal{L}(I_{\ell'})= \emptyset$, so we have to prove that the latter holds for any $\ell \neq \ell'$. \medskip

Now, Proposition~\ref{th:orb.to.proj} implies that the Hamiltonian vector fields of $g_1$ and $g_2$ are orbitally diffeomorphic near any ample covector. And, from Proposition~\ref{Aprop}, in the coordinates $(u_1,\dots,u_n)$ associated with the frame $\{X_1,\dots, X_n\}$, the orbital diffeomorphism satisfies the fundamental algebraic system $A \widetilde{\Phi} = b$.

Let us compute first the matrix $b$. Fix $\ell \in \{1, \dots, N\}$ and $j \in I_\ell$. Using~\eqref{eq:nilp} and the fact that the $\alpha_i$'s are constant, there holds
$$
b^1_j = \frac{\alpha_\ell^2}{\aa} \vec{h}_1(u_j) = \frac{\alpha_\ell^2}{\aa} \sum_{k = m+1}^{n} a^1_{j,k} u_k, \qquad \hbox{and} \qquad b^{s+1}_{j} = \vec{h}_1(b^s_{j}) \ \hbox{ for } s\geq 1.
$$
An easy induction argument gives the value
$$
b_j^s  = \frac{\alpha^2_\ell}{\aa} \sum_{k=m+1}^{n} a^s_{j,k} u_k, \quad s \in \mathbb{N}.
$$
Thus the system of equations $A \widetilde{\Phi} = b$ can be rewritten as
$$
 \sum_{k=m+1}^{n} a^s_{j,k}  \Phi_k = \frac{\alpha^2_\ell}{\aa} \sum_{k=m+1}^{n} a^s_{j,k}  u_k  \qquad \hbox{for every } j\in I_\ell, \ \ell \in \{1,\dots, N\}, \ s \in \N.
$$
In other terms, $A \widetilde{\Phi} = b$ splits into $N$ systems of equations indexed by $\ell =1,\dots, N$ of the form
\begin{equation}
\label{eq:Apsi}
 \sum_{k=m+1}^{n} a^s_{j,k} \left( \Phi_k - \frac{\alpha^2_\ell}{\aa} u_k \right)= 0 \qquad \hbox{for every } j\in I_\ell, \ s \in \N.
\end{equation}

Let us have a closer look to the coefficients $a^s_{j,k}$. Fix as before $\ell \in \{1, \dots, N\}$ and $j \in I_\ell$. First we have
$$
a^1_{j,k} = q_{jk} = \sum_{i=1}^m c_{ij}^k u_i = \sum_{i \in I_\ell} c_{ij}^k u_i,
$$
due to \eqref{eq:struct_coeff}. Using again the latter relation and the other properties of the structure coefficients, an easy induction argument shows that the recurrence formula for $a^s_{j,k}$, $s \in \N$, is
$$
a^{s+1}_{j,k} = \sum_{i \in I_\ell} u_i \vec{u}_i(a^s_{j,k}) + \sum_{l \in \mathcal{L}(I_\ell)} a^s_{j,l} \sum_{i \in I_\ell} c_{il}^k u_i.
$$
As a consequence of this formula:
\begin{itemize}
  \item if $k \not\in \mathcal{L}(I_\ell)$, then $a^{s}_{j,k}=0$; hence, for a fixed $\ell \in \{1,\dots, N\}$, the system~\eqref{eq:Apsi} writes as
  \begin{equation}
  \label{eq:Iell}
   \sum_{k\in \mathcal{L}(I_\ell)} a^s_{j,k} \left( \Phi_k - \frac{\alpha^2_\ell}{\aa} u_k \right)= 0 \qquad \hbox{for every } j\in I_\ell, \ s \in \N;
   \end{equation}
\item if $k \in \mathcal{L}(I_\ell)$, then $a^{s}_{j,k}$ is the corresponding coefficient of the matrix $A$ associated with the family of vector fields $\{X_i, i \in \mathcal{L}(I_\ell)\}$; from Proposition~\ref{th:nonzer.min}, the latter matrix has maximal rank for almost every $u$, thus \eqref{eq:Iell} implies
    \begin{equation}
    \label{eq:valuePhiell}
     \Phi_k = \frac{\alpha^2_\ell}{\aa} u_k   \qquad \hbox{for every } k \in I_\ell.
    \end{equation}
\end{itemize}

Now, assume that there exists two indices $\ell, \ell'$ in $\{1,\dots, N\}$ such that the intersection $\mathcal{L}(I_\ell) \cap \mathcal{L}(I_{\ell'})$ is non empty. For $k \in \mathcal{L}(I_\ell) \cap \mathcal{L}(I_{\ell'})$, we have from~\eqref{eq:valuePhiell}
$$
\Phi_k = \frac{\alpha^2_\ell}{\aa} u_k = \frac{\alpha^2_{\ell'}}{\aa} u_k,
$$
which implies $\ell=\ell'$. Hence $\mathcal{L}(I_\ell) \cap \mathcal{L}(I_{\ell'})= \emptyset$ for any $\ell \neq \ell'$, which implies that $\mathfrak{g}$ is decomposed into a direct sum $\mathfrak{g} = \mathfrak{g}_1 \oplus \cdots \oplus \mathfrak{g}_N$. We conclude that the distribution $D$ admits a product structure $D=D_1 \times \cdots \times D_N$ which is non-trivial since $N>1$. This proves the first part of the theorem.\medskip

It remains to prove that $(g_1, g_2)$ form a Levi-Civita pair on $D$. Set $n_\ell = \dim \mathfrak{g}_\ell$ for $\ell=1,\dots,N$ and define coordinates $x=(\bar x_1,\ldots,\bar x_N)$ on $\mathbb{G}$, where $\bar x_\ell=(x_\ell^1,\ldots, x_\ell^{n_\ell})$, by
$$
x \mapsto \exp\left( \sum_{i \in \mathcal{L}(I_1)} x_1^i X_i\right) \cdots \exp\left( \sum_{i \in \mathcal{L}(I_N)} x_N^i X_i\right).
$$
In these coordinates, a vector field $X_i$ with $i \in \mathcal{L}(I_\ell)$, $\ell=1,\dots,N$,  depends only on the coordinates $\bar x_\ell$ and can be considered as a vector field on $\R^{n_\ell}$ (with coordinates $\bar x_\ell$). Thus $D_\ell$ can be identified with a distribution on $\R^{n_\ell}$. Let $\bar g_\ell$ be the sub-Riemannian metric on $(\R^{n_\ell},D_\ell)$ for which the vector fields $X_i$, $i \in I_\ell$, form an orthonormal frame. We have the following expressions in coordinates:
$$
\left\{
\begin{array}{l}
\displaystyle  g_1 (x)(\dot x, \dot x) =\sum_{\ell=1}^N  \bar g_\ell(\bar{x}_\ell)(\dot{\bar{x}}_\ell, \dot{\bar{x}}_\ell), \\
\displaystyle  g_2 (x)(\dot x, \dot x) =\sum_{\ell=1}^N \alpha^2_\ell  \: \bar g_\ell(\bar{x}_\ell)(\dot{\bar{x}}_\ell, \dot{\bar{x}}_\ell).
\end{array}
\right.
$$
Hence $g_1, g_2$ form a Levi-Civita pair on $D$ with constant coefficients and the theorem is proved.
\end{proof}

\begin{rk}
\label{re:Carnot_micro}
Note that we use the hypothesis of projective equivalence between $g_1$ and $g_2$ only to deduce the existence of a solution to the fundamental algebraic system. So we have actually proved a stronger result than Theorem~\ref{th:Carnot.prod}, namely: \emph{if $g_1$ and $g_2$ are non proportional and if the corresponding fundamental algebraic system $A \widetilde{\Phi} = b$ admits a solution near some $\lambda \in T^*M$, then $D_\mathbb{G}$ admits a nontrivial product structure and $(g_1, g_2)$ is a Levi-Civita pair with constant coefficients.}
\end{rk}


\section{Nilpotent approximation of equivalent metrics}
\label{se:nilp=LCpairs}
\subsection{Nilpotent approximation}
\label{se:nilp_approx}

Let $(M, D, g)$ be a sub-Riemannian manifold and $q_0 \in M$ be a regular point. The nilpotent approximation of $(M, D, g)$ at $q_0$ is another sub-Riemannian manifold, denoted by $(\hat{M}, \hat{D}, \hat{g})$, which has a particular structure: $\hat{M}$ is a Carnot group, $\hat{D}=D_{\hat{M}}$ is the canonical distribution on $\hat{M}$, and $\hat{g}$ is a left-invariant sub-Riemannian metric on $(\hat{M}, \hat{D})$.

Below we briefly recall the construction of the nilpotent approximation in a form convenient for us here, following the foundational paper \cite{tan1} in nilpotent differential geometry. For equivalent description using privileged coordinates or metric tangent space approach see \cite{bel, jea14}.

Let $V^1=D(q_0)$ and, for an integer $i>1$, $V^i= D^i(q_0)/D^{i-1}(q_0)$. The graded space
$$
\mathfrak{g}=\bigoplus_{i=1}^r V^i
$$
associated with the filtration \eqref{flag} at $q_0$ is endowed with the natural structure of a fundamental graded Lie algebra: if $X\in V^i$ and $Y \in V^j$,  then for any vector fields  $\widetilde X$ and $\widetilde Y$ tangent to $D^i$ and $D^j$ respectively in a neighborhood of $q_0$ and such that
$\widetilde X(q_0)=X$, $\widetilde Y(q_0)=Y$, the vector  $[\widetilde X,\widetilde Y](q_0)$ is well-defined modulo $D^{i+j-1}(q_0)$, i.e.\ $[X, Y]:=[\widetilde X,\widetilde Y](q_0)$ is a well-defined element of $V^{i+j}$. The graded Lie algebra $\mathfrak g$ is called the \emph{Tanaka symbol of the distribution $D$ at $q_0$}. Note that since $D$ generates the weak derived flag \eqref{flag}, the space $V^1$ generates the Lie algebra $\mathfrak g$. Therefore, $\mathfrak g$ is a fundamental graded Lie algebra. As a consequence, the connected simply-connected Lie group $\hat M$ with Lie algebra $\mathfrak g$ is a Carnot group.

Let us denote by $\hat{D}$ the left-invariant distribution on $\hat M$ such that $\hat D (e)=V^1$, where $e$ is the identity of $\hat M$. The metric $g$ on $D$ induces an inner product $g(q_0)$ on $V^1$, and so a left-invariant sub-Riemannian metric $\hat g$ on $(\hat{M}, \hat{D})$. The constructed sub-Riemannian manifold $(\hat M, \hat D, \hat g)$ is called the \emph{nilpotent approximation of $(M, D, g)$ at $q_0$}.

Consider a frame $\{X_1, \dots, X_n\}$ of $TM$ adapted to $D$ at $q_0\in M$ and such that $X_1,\dots, X_m$ are $g$-orthonormal. For every $i \in \{1,\dots,n\}$, $X_i(q_0)$ can be identified by the construction above to an element of $\mathfrak g$, which defines a left-invariant vector field $\hat{X}_i$ on $\hat M$. Then $\hat{X}_1, \dots, \hat{X}_m$ are $\hat g$-orthonormal and $\{\hat{X}_1, \dots, \hat{X}_n\}$ is a frame of $T \hat M$ adapted to $\hat D$ at any point of $\hat M$. The structure coefficients $\hat{c}_{ij}^k$ of this frame satisfy:
\begin{equation}
\label{struct.constant.nilpotent}
\left\{
\begin{array}{ll}
\hat{c}_{ij}^k  \equiv c_{ij}^k(q_0) \ &\hbox{ if } \ w_i + w_j = w_k; \\
\hat{c}_{ij}^k \equiv 0 \ &\hbox{ if } \ w_i + w_j \neq w_k.
\end{array}
\right.
\end{equation}

\subsection{Equivalence for nilpotent approximations} \label{subsec:nilp.eqv}
Let $(M,D,g_1)$ and $(M,D,g_2)$ be two sub-Riemannian manifolds. We fix a point $q_0$ which is stable with respect to $g_1,g_2$ and we denote by  $(\hat{M}, \hat{D}, \hat{g}_i)$, $i = 1, 2$, the nilpotent approximation of $(M, D, g_i)$ at $q_0$.

\begin{Theorem}
\label{th:nilpotent_equiv}
If $g_1, g_2$ are projectively equivalent and not conformal to each other near $q_0$, then $\hat{D}$ admits a product structure and $(\hat{g}_1, \hat{g}_2)$ is a Levi-Civita pair with constant coefficients.
\end{Theorem}

To prove this result we need first some technical results.

Let $g_1, g_2$ be two non-trivially projectively equivalent metrics. By Proposition~\ref{th:orb.to.proj}, their Hamiltonian vector fields are orbitally diffeomorphic near any ample covector. We choose a frame $\{X_1, \dots, X_n\}$ adapted to $(g_1,g_2)$ near $q_0$. It induces (see subsection~\ref{se:nilp_approx}) a frame $\{\hat{X}_1, \dots, \hat{X}_n\}$ of $T \hat M$ adapted to $\hat D$ which has by construction the following properties:
$\hat{X}_1, \dots, \hat{X}_m$ is $\hat{g}_1$-orthonormal and $\frac{1}{\alpha_1(q_0)}\hat{X}_1, \dots, \frac{1}{\alpha_m(q_0)}\hat{X}_m$ is $\hat{g}_2$-orthonormal, where $\alpha^2_1(q), \dots, \alpha^2_m(q)$ are the eigenvalues of the transition operator at $q$ between $g_1$ and $g_2$. Note that the transition operator between $\hat{g}_1$ and $\hat{g}_2$ has the same eigenvalues $\alpha^2_1(q_0), \dots, \alpha^2_m(q_0)$ at any point of $\hat M$.

\begin{rk}
\label{re:non_conformal}
The metrics $\hat{g}_1$ and $\hat{g}_2$ are proportional if and only if all $\alpha^2_i(q_0)$'s are equal, i.e.\ if $g_1, g_2$ are conformal to one another near $q_0$ (recall that $q_0$ is stable). The hypothesis of the theorem rules out this possibility.
\end{rk}

Recall that the data of a frame $\{X_1, \dots, X_n\}$ of $TM$ and of eigenvalues $\alpha^2_1, \dots, \alpha^2_m$ allows to construct infinite matrices $A$ and $b$ by the formulas \eqref{eq:A_and_b}--\eqref{term.of.b.rec.form}. Each element of these matrices $A=A(q)(u)$ and $b=b(q)(u)$ is a function of $q$ in a neighborhood of $q_0$ and of $u \in \R^n$. Similarly, denote by $\hat A$ and $\hat b$ the matrices constructed by using $\{\hat{X}_1, \dots, \hat{X}_n\}$ as a frame and $\alpha^2_1(q_0), \dots, \alpha^2_m(q_0)$ as eigenvalues in the formulas \eqref{eq:A_and_b}--\eqref{term.of.b.rec.form}. Each element of  $\hat A= \hat A (\hat q)(u)$ and $\hat{b}=\hat{b}(\hat{q})(u)$ is a function of $\hat{q}$ in $\hat M$ and of $u \in \R^n$. Finally,  the elements of the matrices $A$, $b$, $\hat A$ and $\hat b$ are denoted by $a^s_{j,k}$, $b^s_j$, $\hat a^s_{j,k}$ and $\hat b^s_j$ respectively.
Let us introduce the notion of \emph{weighted degree} $\mathrm{deg}_w$ for a polynomial with $n$ variables. For a monomial $m=u_1^{\beta_1} \cdots u_n^{\beta_n}$, we set $\mathrm{deg}_w(m) = \sum_{i=1}^{n} \beta_i w_i$. Then the weighted degree $\mathrm{deg}_w(P)$ of a polynomial function $P=P(u_1,\dots,u_n)$ is the largest weighted degree of the monomials of $P$. A polynomial is said to be \emph{$w$-homogeneous} if all its monomials are of the same weighted degree.

\begin{Lemma}
\label{le:degA}
For any $s \in \N$, $1 \leq j \leq m$, and $m+1 \leq k \leq n$, there hold:
\begin{itemize}
\item for every $q \in M$ near $q_0$, the element $a^s_{j,k}(q)$ is a polynomial in $u_1,\dots,u_n$ of weighted degree
$$
\mathrm{deg}_w (a^s_{j,k}(q)) \leq 2s - w_k + 1;
$$

\item the function $\hat{a}^s_{j,k}$ does not depend on $\hat q  \in \hat M$ and is a $w$-homogeneous polynomial in $u_1,\dots,u_n$ of weighted degree
	\begin{equation*}
	\mathrm{deg}_w (\hat{a}^s_{j,k}) = 2s - w_k + 1;
	\end{equation*}

\item the homogeneous term of highest weighted degree in $a^s_{j,k}(q_0)$ is $\hat{a}^s_{j,k}$, that is,
\begin{equation*}
	a^s_{j,k}(q_0)(u)= \hat{a}^s_{j,k} (u) + \mathrm{poly}(u_1, \dots, u_n),
\end{equation*}
with $\mathrm{deg}_w(\mathrm{poly}) < \mathrm{deg}_w(\hat{a}^s_{j,k})$.
\end{itemize}
\end{Lemma}

\begin{proof}
Notice first that a structure coefficient $c_{ij}^l$ is zero if $w_l>w_i + w_j$; and second that, for any polynomial $P$,
\begin{equation*}
\label{+2}
\mathrm{deg}_w\bigl(\vec h_1(P)\bigr) \leq  \mathrm{deg}_w(P)+2.
\end{equation*}
An easy induction argument based on \eqref{elem.A} allows then to prove the first item.
The second and the third item are proven in the same way by using moreover \eqref{struct.constant.nilpotent}.
\end{proof}

\begin{Lemma}
\label{le:degb}
For any $s \in \N$ and $1 \leq j \leq m$, there hold:
\begin{itemize}
\item for every $q \in M$ near $q_0$,  $\aa b^s_j$ is a polynomial in $u_1,\dots,u_n$ of weighted degree
$$
\mathrm{deg}_w (\aa b^s_j) \leq 2s + 1;
$$

\item the function $\aa (q_0) \hat{b}^s_j$ does not depend on $\hat q  \in \hat M$ and is a $w$-homogeneous polynomial in $u_1,\dots,u_n$ of weighted degree
	\begin{equation*}
	\mathrm{deg}_w (\aa (q_0) \hat{b}^s_j) = 2s + 1;
	\end{equation*}

\item the homogeneous term of highest weighted degree in $b^s_j(q_0)$ is $\hat{b}^s_j$, that is,
\begin{equation*}
	\aa b^s_j(q_0)(u)= \aa \hat{b}^s_j (u) + \mathrm{poly}(u_1, \dots, u_n),
\end{equation*}
with $\mathrm{deg}_w(\mathrm{poly}) < \mathrm{deg}_w(\hat{b}^s_j)$.
\end{itemize}
\end{Lemma}

\begin{proof}
Note first that, by \eqref{eq:h1P} and Lemma~\ref{le:PQ}, $\vec{h}_1 (\aa^2) / \aa^2 = Q$. Thus it is a polynomial function of $u$ of weighted degree 1.  As a consequence, the terms $R_j$, $j=1, \dots, m$, are polynomials of weighted degree 3. Using then the recurrence formula \eqref{term.of.b.rec.form} and the fact that
$$
\aa \vec{h}_1 (\frac{1}{\aa})= - \frac{1}{2} \frac{\vec{h}_1 (\aa^2)}{\aa^2} = - \frac{Q}{2},
$$
an easy induction argument shows the first item.

The second and the third item are proven in the same way by using moreover \eqref{struct.constant.nilpotent}.
\end{proof}

\begin{Lemma}
\label{le:minors}
Assume that a minor of the matrix $\hat{A}$ (resp.\ of the matrix $\begin{pmatrix} \hat{A} & \aa(q_0) \hat b \end{pmatrix}$) is nonzero. Then the corresponding minor - same rows and columns - of $A$ (resp.\ of $\begin{pmatrix} A  & \aa b\end{pmatrix}$) is nonzero as well near $q_0$.
\end{Lemma}

\begin{proof}
An arbitrary $(l \times l)$ minor $m(A)$ of the matrix $A$ has the form
$$
m(A) = \sum_{\sigma \in \mathfrak{S}_l} \mathrm{sgn}(\sigma)\  a^{s_1}_{j_1, k_{\sigma(1)}} \cdots a^{s_l}_{j_l, k_{\sigma(l)}}.
$$
As a consequence of Lemma~\ref{le:degA}, each term in this sum is a polynomial function of $u$ of weighted degree $\leq 2 \sum_i s_i - \sum_i w_{k_i} + l$. Moreover, the homogeneous part of $m(A(q_0))$ of weighted degree $2 \sum_i s_i - \sum_i w_{k_i} + l$ is equal to
$$
m(\hat{A}) = \sum_{\sigma \in \mathfrak{S}_l} \mathrm{sgn}(\sigma)\  \hat{a}^{s_1}_{j_1, k_{\sigma(1)}} \cdots \hat{a}^{s_l}_{j_l, k_{\sigma(l)}}.
$$
Hence, if $m(\hat{A}) \neq 0$, then $m(A(q_0))$ is nonzero and so is $m(A(q))$ for $q$ near $q_0$.

The same argument holds for the minors of $\begin{pmatrix} \hat{A} & \aa(q_0) \hat b \end{pmatrix}$ and $\begin{pmatrix} A  & \aa b\end{pmatrix}$ by using Lemma~\ref{le:degb}.
\end{proof}

\begin{Lemma}
\label{prop:affin.eqv}
The algebraic system $\hat{A} \hat{\Phi}=\hat{b}$ admits a solution $\hat{\Phi}$ near any ample covector in $\pi^{-1}(q_0)$.
\end{Lemma}

\begin{proof}
Since $\vec{h}_1$ and $\vec{h}_2$ are locally orbitally diffeomorphic near an ample covector, there exists an orbital diffeomorphism $\Phi$ between the extremal flows of $(g_1,g_2)$ with coordinates $(\Phi_1, \dots, \Phi_n)$ in the system of coordinates associated with the frame $\{X_1, \dots, X_n\}$. Then from Proposition~\ref{Aprop} $\widetilde{\Phi}=(\Phi_{m+1}, \dots, \Phi_n)$ satisfies $A \widetilde{\Phi} = b$. Introducing the nonzero function $\aa$ defined by \eqref{eq:aa}, this can be rewritten as $A \aa \widetilde{\Phi} - \aa b = 0$, i.e.
$$
\begin{pmatrix} A & \aa b \end{pmatrix} \begin{pmatrix} \aa \widetilde{\Phi} \\ -1 \end{pmatrix} = 0.
$$
Thus $\begin{pmatrix} A & b \end{pmatrix}$ is not of full rank, or equivalently, any maximal minor of the latter matrix is zero.
The contraposition of Lemma~\ref{le:minors} implies that any maximal minor of $\begin{pmatrix} \hat{A} & \aa (q_0)\hat{b} \end{pmatrix}$ is zero as well, thus this matrix is not of full rank.

Any element of $\ker \begin{pmatrix} \hat{A} & \aa (q_0)\hat{b} \end{pmatrix}$ is a function of $u \in \R^n$ with values in $\R^{n-m} \times \R$. Since $\hat A$ is of full rank by Proposition~\ref{th:nonzer.min}, and since $\alpha(q_0)$ is nonzero, there exists $\Psi \in \ker \begin{pmatrix} \hat{A} & \aa (q_0)\hat{b} \end{pmatrix}$ of the form
$\Psi = ( \aa(q_0) \hat{\Phi}, -1)$. In other terms, $\hat{\Phi}$ satisfies $\hat{A} \hat{\Phi}=\hat{b}$.
\end{proof}

\begin{proof}[Proof of Theorem~\ref{th:nilpotent_equiv}]
The metrics $\hat{g}_1$ and $\hat g_2$ are left-invariant metrics on the Carnot group $\hat M$ and by Remark~\ref{re:non_conformal} they are not proportional. Moreover by the lemma above, the fundamental algebraic system associated with  $\hat{g}_1$ and $\hat g_2$ admits a solution. Theorem~\ref{th:nilpotent_equiv} follows then from Remark~\ref{re:Carnot_micro}.
\end{proof}


 \section{Genericity of indecomposable fundamental graded Lie algebras}
\label{sec:gen}
This section is dedicated to the proof of Theorem~\ref{th:generic_distrib}.
From Theorem~\ref{th:nilpotent_equiv}, the existence of projectively equivalent and non conformal metrics implies that the nilpotent approximation of $D$ at generic points admits a product structure. Thus we have to show that, under the hypothesis of the theorem on $(m,n)$, generic nilpotent approximations do not have a product structure.

Remark first that, when $n \geq \frac{m(m+1)}{2}$,
generic distributions are free up to the second step at generic points, i.e.\ $D^2$ is a distribution of rank $\frac{m(m+1)}{2}$ near these points. The nilpotent approximation of such a distribution does not admit a product structure, therefore the statement of Theorem~\ref{th:generic_distrib} holds for these values of $(m,n)$.

Consider now a pair $(m,n)$ such that $2 \leq m < n\leq \frac{m(m+1)}{2}$.
We denote by
$\mathrm{GNLA}(m,n)$ the set of all $n$-dimensional step $2$ graded Lie
algebras generated by the homogeneous component $V_1$ of dimension $m$. Theorem~\ref{th:generic_distrib} results directly from the following proposition.

\begin{Prop}
	\label{allgenTan}
Except the following two cases:
\begin{enumerate}
		\item  $m=n-1$ with even $n$,
		\item $(m,n)=(4,6)$,
	\end{enumerate}
	a generic element of $\mathrm{GNLA}(m,n)$ cannot be represented as a
	direct sum of two graded Lie algebras.
\end{Prop}

\begin{proof}
	Let $\mathfrak g= V_1\oplus V_2$ be a step $2$ graded Lie algebra. This
	algebra can be described  by the \emph{Levi operator} $$\mathcal L_q:\wedge^2
	V_1\rightarrow V_2,$$ which sends $(X, Y)\in \wedge^2 V_1$ to $[X,Y]$  or,
	equivalently, by the dual operator $\mathcal L_q^*: V_2^*\rightarrow \wedge^2
	V_1^*$.
	Denote by $\Omega_\mathfrak g$ the image of the latter operator.
	
	Since $\mathfrak{g}$ is generated by $V_1$, the space $\Omega_\mathfrak g$ is a
	$(n-m)$-dimensional subspace of the space $\wedge^2 V_1^*$  of all skew-symmetric
	forms on $V_1$.
	The set  $\mathrm{GNLA}(m,n)$ is
	in a bijective correspondence with the orbits of  $(n-m)$-dimensional  subspace
	of $\wedge^2 V_1^*$ under the natural action of $GL(V_1)$.
	This reduces our question to an analysis of orbits in Grassmannians of $\wedge^2
	V_1^*$ under the natural action of $GL(V_1)$.

Given a subspace $W$ of $V_1$ denote by $A_W$ the space of all skew-symmetric forms with kernel $W$.
A graded Lie algebra $\mathfrak g=V_1\oplus V_2$ is a direct sum of two graded
	Lie algebras if and only if there is a splitting
	\begin{equation}
	\label{split1}
	V_1=V_1^1\oplus V_1^2
	\end{equation}
	(with each summand being nonzero)  such that the corresponding subspace
	of $\wedge^2 V_1^*$  can be represented as
	\begin{equation}
	\label{split2}
	\Omega_{\mathfrak g}=\Omega_{\mathfrak g}^1\oplus \Omega_{\mathfrak g}^2, \qquad \Omega_{\mathfrak{g}}^1\subset A_{V_1^2}, \quad  \Omega_{\mathfrak{g}}^2 \subset A_{V_1^1}.
	\end{equation}
	In this case we will say that the
	space $\Omega_{\mathfrak{g}}$ is \emph{decomposable} with respect to the splitting \eqref{split1}.
The condition on $\Omega_{\mathfrak{g}}^1$ and $\Omega_{\mathfrak{g}}^2$ in \eqref{split2} is equivalent to require that, in some basis of $V_1$, the elements of $\Omega_{\mathfrak{g}}^1$ have the
	block-diagonal matrix representation
	$ \left(
	\begin{array}{cc}
	A_1 &\multicolumn{1}{|c}{0}\\ \hline
	0 & \multicolumn{1}{|c}{0}
	\end{array}\right)$
	and the elements of $\Omega_{\mathfrak{g}}^2$ have the block-diagonal matrix representation
	$ \left(
	\begin{array}{cc}
	0 &\multicolumn{1}{|c}{0}\\ \hline
	0 & \multicolumn{1}{|c}{A_2}
	\end{array}\right)$,
	where the corresponding blocks have the same nonzero size. Note that we do
	not exclude that one of the subspaces $\Omega_{\mathfrak{g}}^i$ is equal to zero. In this case the space $\Omega_{\mathfrak{g}}$ itself must consist of forms with a common nontrivial kernel.
	This corresponds to the situation where one of the summands in the decomposition of $\mathfrak{g}$ into a direct
	sum is commutative.

We will distinguish several cases, depending on the value of the corank $n-m$ and on the parity of $m$. \smallskip

\noindent	{\bf 1. The case $n-m=1$.}
	In this case the space $\Omega_{\mathfrak{g}}$ is
	a line in the space $\wedge^2 V_1^*$. It is decomposable if and only if $\Omega_{\mathfrak{g}}$ is generated by a degenerate skew-symmetric form, one of the subspaces $\Omega_{\mathfrak{g}}^i$ being zero. The latter condition is satisfied by a generic line in $\wedge^2 V_1^*$ if and only if $\dim V_1$ is odd, or equivalently, when $n$ is even.\smallskip

\noindent	{\bf 2. The case $n-m=2$.}
	In this case  $\Omega_{\mathfrak{g}}$  is
	a plane in the space $\wedge^2 V_1^*$. The orbits of planes
	of $\wedge^2 V_1^*$ under the natural action of $GL(V_1)$ are in bijective correspondence with the
	equivalence
	classes of  \emph{pencils of skew-symmetric 2-forms}, which are linear combinations $\lambda A + \mu B$ of two skew-symmetric 2-forms $A, B$ with real parameters $\lambda, \mu$. The classification of these pencils is
	classical, we give here some basic definitions and results and we refer the reader to \cite{gauger} (based on \cite{gantmacher}) for more details.
	
	Let us consider a pencil of skew-symmetric 2-forms $\lambda A + \mu B$, identified to a pencil of skew-symmetric matrices in some basis of the space $V_1$. The pencil is called \textit{regular} if its determinant is a non-zero polynomial, it is called \textit{singular} otherwise.  A regular pencil is characterized by  its \textit{elementary divisors}, defined as follows. Consider the greatest common divisor of all rank-$k$ minors of the pencil for the integers $k$ for which it makes sense. The elementary divisors of the pencil are the simple factors (with their multiplicity) of these greatest common divisors for all possible $k$. In case of skew-symmetric pencils, all elementary divisors come in pairs. A singular pencil is characterized by its elementary divisors and its \textit{minimal indices} (also called {Kronecker indices} in \cite{gauger}). The special property of a singular pencil is that there exists a nonzero homogeneous polynomial branch of kernels $\lambda, \mu \mapsto v(\lambda, \mu)$, i.e.\ for any $\lambda, \mu \in \mathbb R$, the vector $v(\lambda, \mu)$ is a nonzero element of  $\ker (\lambda A + \mu B)$. The \textit{first minimal index}  is the minimal possible degree of a polynomial $v(\lambda, \mu)$. We do not need the other indices here, so we do not define them.

The pencils defined in different basis of $V_1$ and associated to different elements of the same $GL(V_1)$-orbit of a skew-symmetric form are called equivalent. The following result give the normal forms of skew-symmetric pencils.
	
	\begin{Theorem}[\cite{gauger}]
\label{th:norm.form}
		A skew-symmetric pencil $\lambda A + \mu B$ with minimal indices $m_1 \leq m_2 \leq \cdots \leq m_p$ and elementary divisors $(\mu + a_1 \lambda)^{l_1}, (\mu + a_1 \lambda)^{l_1}, \dots, (\mu + a_q \lambda)^{l_q}, (\mu + a_q \lambda)^{l_q}, (\lambda)^{f_1}, (\lambda)^{f_1}, \dots, (\lambda)^{f_s}, (\lambda)^{f_s}$ is equivalent to the skew-symmetric pencil $Q$ of the following form,
		\begin{equation}
		Q =  \left(
		\begin{array}{cc}
		\mathcal{M} &\multicolumn{1}{|c}{0}\\ \hline
		0 & \multicolumn{1}{|c}{\mathcal{F}}
		\end{array}\right) ,
		\end{equation}
where the singular part $\mathcal{M}$ and the regular part $\mathcal{F}$ satisfy
		\begin{equation}\mathcal{M}=
		\begin{pmatrix}
		M_{1} & & \\
		& \ddots & \\
		& & M_{q}
		\end{pmatrix},
		\qquad
		\mathcal{F} =
		\begin{pmatrix}
		E_{1}(a_1) & & & & &\\
		& \ddots & & & &\\
		& & E_{q}(a_q) & & &\\
		& & & F_1 & & \\
		& & & & \ddots  & \\
		& & & & & F_s
		\end{pmatrix} ,
		\end{equation}
all blocks $M_i, E_i, F_i$ being skew-symmetric, $M_i$ of size $(2m_i+1)\times(2m_i+1)$, $E_i$ of  size $2l_i\times2l_i$ and $F_i$ of  size $2f_i\times 2f_i$.
\end{Theorem}

\begin{rk}
Note that the only possible zero blocks are the blocks $E_i(a_i)$ with $\frac{\mu}{\lambda} = - a_i$ and $l_i = 1$, and $F_j$ with $\lambda = 0$ and $f_j = 1$.
\end{rk}

Let us return to the plane $\Omega_g$ considered as a pencil. 	The cases of odd-dimensional and even-dimensional $V_1$ are treated again separately.\smallskip
	
\noindent	{\bf 2(a) The subcase when $\dim V_1$ is odd, $\dim V_1=2k+1$.}
	In this case all forms in the pencil $\Omega_g$ are degenerate, so the pencil is singular. From the dimension of the blocks in the normal form, we see that the first minimal index is not greater than $k$. Moreover, for generic pencils this first minimal index has its maximal possible value, thus it is equal to $k$.
	
	On the other hand, if the pencil $\Omega_{\mathfrak g}$ is decomposable, then its first minimal index must be equal to zero, i.e.\ all forms of the pencil have a common nontrivial kernel. Indeed, assume that  $\Omega_{\mathfrak g}$ is decomposable with respect to the splitting \eqref{split1} with decomposition \eqref{split2}.
	The statement is clear if one of the spaces $\Omega_{\mathfrak g}^i$ in \eqref{split2}  is zero. The remaining possibility is that the spaces $\Omega_{\mathfrak{g}}^i$ are both  one-dimensional. Without loss of generality we assume that $V_1^1$ is odd-dimensional. Then all forms on the line  $\Omega_{\mathfrak{g}}^1$ have a nontrivial kernel in $V_1^1$ and this kernel is common for all forms in $\Omega_{\mathfrak g}$.
	
	Since $k>0$ ($\dim V_1 \geq 2$), we conclude that generic pencils are not decomposable.\smallskip
	
\noindent	{\bf 2(b) The subcase when $\dim V_1$ is even.}
	In this case generic pencils are regular. Generic regular pencils of skew-symmetric forms have only simple elementary divisors, i.e.\ linear and not nontrivial powers of linear, such that each divisor appear only twice.
	
	Now consider a decomposable regular pencil  $\Omega_{\mathfrak g}$ with respect to the splitting \eqref{split1} with decomposition \eqref{split2}. Then	
\begin{equation}
\label{split3}
\Omega_g=\{\lambda  \omega_1+\mu  \omega_2: \lambda, \mu\in \mathbb R\},
\end{equation}
where the form $\omega_i$ generates $\Omega_{\mathfrak g}^i$, $i=1,2$. One can see by the normal form in Theorem \ref{th:norm.form} that the elementary divisors of this pencil can be only of the form $\lambda$ or $\mu$ of multiplicity one. The pencil satisfy the genericity property of the previous paragraph if and only if the set of elementary divisors is $\{\lambda, \lambda, \mu, \mu\}$, i.e.\ when $m=4$. Consequently $n=6$. So, decomposibility on an open set can occur only if  $(m,n)=(4,6)$.
Conversely, if $(m,n)=(4,6)$ and $\Omega _g$ is as in \eqref{split3}, then ${\rm Pfaffian}\bigl(\lambda  \omega_1+\mu  \omega_2\bigr)$ is a quadratic form in $\lambda$ and $\mu$ and the pencil $\Omega_g$ is decomposible if and only if this form is sign-indefinite, which implies that decomposibility occurs on an open set in this case.
\smallskip

\begin{rk}
\label{kerrem}
In the case $m=4$, if the pencil $\Omega_g=\{\lambda  \omega_1+\mu  \omega_2 : \lambda, \mu \in \mathbb R \}$ is decomposable, then  the subspaces $V_1^1$ and $V_{1}^2$ in the splitting \eqref{split1}  are uniquely defined. Indeed,  in this case there are exactly two degenerate forms: these are the lines on which ${\rm Pfaffian}(\lambda  \omega_1+\mu  \omega_2)=0$ and the subspaces  $V_i^j$ are kernels of these forms. We will call this splitting $V_1=V_1^1\oplus V_{1}^2$ the canonical splitting corresponding to the decomposable pencil $\Omega_g$.
\end{rk}
	
\noindent	{\bf 3. The case $n-m>2$.}
	We will reduce this case to the case $n-m=2$.\smallskip
	
\noindent	{\bf 3(a) The subcase when $\dim V_1$ is odd, i.e.\ $\dim V_1 = 2k+1$}.
	Assume that $\Omega_{\mathfrak g}$ is decomposable with respect to the splitting \eqref{split1} and, without loss of generality, that $\dim V_1^1$ is odd and  equal to $2l+1$, $l<k$. Then it is easy to see on the normal form that the first minimal index of any plane in $\Omega_{\mathfrak g}$ is not greater than $l$. On the other hand, a generic plane in a generic $(n-m)$-dimensional subspace of $\wedge^2 V_1^*$ has first minimal index $k$. This proves the statement of the theorem in this case.\smallskip

\noindent	{\bf 3(b) The subcase when $\dim V_1$ is even.}
First, assume that $m=\dim V_1>4$. Then by item 2(b) a generic $(n-m)$-dimensional subspace of $\wedge ^2 V_1$ contains an indecomposable plane, therefore the original $(n-m)$-dimensional subspace is indecomposable.

Now assume that $m=\dim V_1=4$. Then a generic $(n-m)$-dimensional subspace of $\wedge ^2 V_1$ either contains an indecomposable plane or contains two planes such that the canonical decomposition of $V_1$ corresponding to these planes, as defined in Remark \ref{kerrem},  do not coincide.  This implies that generic $(n-m)$-dimensional subspaces of $\wedge^2 V_1$ are indecomposable. This case ends the proof.
\end{proof}

\appendix


\section{Proof of Proposition~\ref{le:LCpairs} on projective equivalence of Levi-Civita pairs}
\label{se:proof_LCpairs}
Let $(g_1,g_2)$ be a Levi-Civita pair on a distribution $D$ of a manifold $M$,
and  fix a point $q_0 \in M$. In local coordinates, the metrics $g_1, g_2$ take the form~\eqref{met1} and the distribution $D$ is the product distribution $D=D_1 \times \cdots \times D_N$ on $\R^n= \R^{n_1}\times \cdots \times \R^{n_N}$ defined by~\eqref{eq:product_distrib}.
	
	Let us construct a frame adapted to $(g_1, g_2)$. For any integer $1\leq \ell \leq N$, we choose vector fields $Y_1^\ell, \dots, Y^\ell_{k_\ell}$, where $k_\ell = \dim D_\ell$, of the form $Y^\ell_i = \sum_{j=1}^{n_\ell} a^\ell_{ij}(x^\ell)\partial_{x_j^\ell}$ such that $\{ Y_1^\ell, \dots, Y^\ell_{k_\ell}\}$ is a frame of $D_\ell$ and is orthonormal with respect to $\bar{g}_\ell$.
	We complete $\{ Y_1^\ell, \dots, Y^\ell_{k_\ell}\}$ into a frame adapted to the flag $D_\ell \subset D_\ell^2 \subset \dots \subset T\R^{n_\ell}$ by adding vector fields $ X_{k_\ell+1}^\ell, \dots, X^\ell_{n_\ell}$ of the form $[Y^\ell_{i_1}, \dots , [Y_{i_{k-1}}^\ell, Y^\ell_{i_k}]]$. Moreover, setting $X^\ell_i = \frac{1}{\sqrt{\gamma_\ell}}Y^\ell_i$ for $i=1, \dots, k_\ell$, we obtain a $g_1$-orthonormal frame $\{ X_1^\ell, \dots, X^\ell_{k_\ell}\}$ of $D_\ell$.
	
	Grouping all together, we have obtained a frame $\{X^\ell_i, \ 1\leq \ell \leq N, \ i=1, \dots, k_\ell\}$ of $D$ which is $g_1$-orthonormal and $g_2$-orthogonal, and a frame $\{X^\ell_i, \ 1\leq \ell \leq N, \ i=1, \dots, n_\ell\}$ of $T\R^n$ which is adapted to the pair $(g_1, g_2)$. To simplify the notations we denote by $\{X_1,\dots,X_m\}$ the frame of $D$ and by $\{X_1,\dots,X_n\}$ the frame of $T\R^n$. For $i=1,\dots,n$, we denote by $\ell(i)$ the integer in $\{ 1, \dots, N\}$ such that $X_i$ is of the form $X^{\ell(i)}_j$.
	
	The special form of the constructed adapted frame and the form of \eqref{met2} imply the following properties of the structure coefficients $c^k_{ij}$:
	\begin{itemize}
		\item if $\ell(i) \neq \ell(j)$, then $c_{ij}^k=0$ if $k \neq i$ or $j$; moreover,
\begin{equation}
\label{eq:cijk_LC}
c^j_{ij} =\left\{
		\begin{array}{ll} 
		\frac{\alpha^2_{\ell(j)} X_i(\alpha^2_{\ell(i)}) }{4 \alpha_{\ell(i)}^2 \left( \alpha_{\ell(j)}^2 - \alpha_{\ell(i)}^2 \right)}, & \hbox{if } j\leq m; \\[3mm]
		0, & \hbox{if } j>m;
		\end{array}
		\right.
\end{equation}
		
\item if $\ell(i)= \ell(j) \leq \ell(k)$, then $c_{ij}^k=0$.
\end{itemize}
	Notice also that we can obtain the following relationship from~\eqref{met2},
\begin{equation}
\label{eq:X_ialpha_LC}
		X_i(\alpha_{\ell(j)}^2) = \frac{\alpha^2_{\ell(j)}}{2 \alpha_{\ell(i)}^2} X_i(\alpha_{\ell(i)}^2), \qquad \text{if } {\ell(i)} \neq {\ell(j)}.
\end{equation}
	
	These formulas permit us to simplify the equations \eqref{eq:220} and \eqref{eq:221} which characterize an orbital diffeomorphism. To simplify \eqref{eq:220}, we have to compute $R_j$. For this, we first show that the first divisibility condition holds for our choice of adapted frame (it results directly from the use of~\eqref{eq:cijk_LC} and \eqref{eq:X_ialpha_LC} in the computation of $\vec{h}_1(\mathcal{P})$). Then we use the following formula (see \cite[Lemma\ 3]{Z}),
	\begin{equation*}
		\begin{split}
			R_j  = & \sum_{i=1}^{m} (1-\delta_{ij})\left((\alpha_j^2 - \alpha_i^2)c^i_{ji} - \frac{X_j(\alpha_i^2)}{2}\right)u^2_i +  \sum_{i=1}^{m}(1-\delta_{ij})\frac{\alpha_i^2}{2 \alpha_j^2}X_i\left(\frac{\alpha_j^4}{\alpha_i^2}\right)u_i u_j\\
			&+ \sum_{i=1}^{m} \sum_{k=1}^{m} (1-\delta_{ik})(\alpha_j^2 - \alpha_k^2)c^k_{ji}u_i u_k + \alpha_j^2 \sum_{i=1}^{m}\sum_{k=m+1}^{n}c^k_{ji}u_i u_k.
		\end{split}
	\end{equation*}
	We substitute the structure coefficients by the expressions shown above and use the property of functions $\beta_\ell(\bar x_\ell)$ to be constant if $x_\ell$ is of dimension more then one. We get
	$R_j = \alpha_j^2 \sum_{i=1}^{m} \sum_{k = m+1}^{n} c^k_{ij} u_i u_k.$
	We finally obtain a simplified form of \eqref{eq:220},
	\begin{equation} \label{(1.a)}
		\sum_{k=m+1}^{n}q_{jk}\Phi_k = \frac{\alpha_j^2}{\aa} \sum_{k=m+1}^{n} q_{jk} u_k \qquad 1 \leq j \leq m.
	\end{equation}
	To simplify  \eqref{eq:221}, it is sufficient to notice that $X^s_i(\alpha^2_i) = 0$ if  $\left| I_s \right| > 1$. Setting $\Phi_i = \frac{\alpha_i^2 u_i}{\aa}$ for $i = 1 , \dots, m$ as in \eqref{eq:first.m.coord}, we obtain
$$
 \vec{h}_1 (\Phi_s) = \sum_{k=1}^{n} q_{sk} \Phi_k.
 $$

To summarize, there exists an orbital diffeomorphism between $\vec{h}_1$ and $\vec{h}_2$ if the following equations have a solution:
	\begin{equation}
		\sum_{k=m+1}^{n}q_{jk}\Phi_k = \frac{\alpha_j^2}{\aa} \sum_{k=m+1}^{n} q_{jk} u_k \qquad 1 \leq j \leq m,
	\end{equation}
	\begin{equation} \label{(2.a)}
		\vec{h}_1(\Phi_s) = \sum_{k=1}^{n}q_{sk} \Phi_k  \qquad m+1 \leq s \leq n.
	\end{equation}	
It appears that $\Phi_k = \frac{\alpha_k^2u_k}{\aa}$, $k= m+1, \dots, n$, obviously satisfy this system. Thus $\vec{h}_1$ and $\vec{h}_2$ are orbitally diffeomorphic and, by Proposition \ref{th:proj.to.orb}, $g_1, g_2$ are projectively equivalent.
	
In the case of a pair with constant coefficients, all $\alpha_i$ are constant and thus $\vec{h}_1(\alpha^2) = 0$. Applying again Proposition~\ref{th:proj.to.orb}, we deduce that the metrics of a Levi-Civita pair with constant coefficients are affinely equivalent. Conversely, if the metrics of a Levi-Civita pair are affinely equivalent, then by Proposition~\ref{le:constant_alphai} all factors $\alpha_i$ are constant, which implies that all $\beta_i$ are constant. Thus the pair has constant coefficients.
 This ends the proof of Proposition~\ref{le:LCpairs}.


\section{Proof of Proposition~\ref{le:quadratic} on quadratic first-integrals}
\label{se:proof_quadratic}

Proposition~\ref{le:quadratic} is the generalization to sub-Riemannian metrics of a result stated for Riemannian metrics in \cite{Kruglikov-Matveev2015}, namely Corollary 3 of Theorem 1. It is then sufficient to show the following result, which is the exact generalization to the sub-Riemannian case of that Theorem 1 (in the case of polynomials of degree $d=2$).

\begin{Prop}
\label{le:KM}
Let $D$ be a Lie-bracket generating distribution on an open ball $B \subset \R^n$ and $g$ be a smooth metric on $D$. Then, for any $\eps >0$ there exists a metric $\widetilde{g}$ on $D$ which is $\eps$-close to $g$ in the $C^\infty$-topology, and $\eps'>0$ such that for any $C^2$ metric $g'$ on $D$ which is $\eps'$-close to $\widetilde{g}$ in the $C^2$-topology, the normal extremal flow of $g'$ does not admit a non-trivial quadratic first-integral (non-trivial means non proportional to the Hamiltonian $h_{g'}$ associated with $g'$).
\end{Prop}

Note that we work on an open subset of $\R^n$ and not in a general manifold since, as noticed in \cite{Kruglikov-Matveev2015}, it is sufficient to prove the result locally. Thus we identify $T^*B$ to $B \times \R^n$ and we write a covector $\lambda \in T^*B$ as a pair $(x,p)$, where $x=\pi(\lambda)$.

The proof of Theorem 1 in \cite{Kruglikov-Matveev2015} goes as follows.
Choose $k$ sets of $N$ points\footnote{In \cite{Kruglikov-Matveev2015}, the sets of points are labelled $A=\{A_1, \dots, A_N\}$, $B_\ell=\{B_{\ell,1}, \dots, B_{\ell,N}\}$, $\ell=1, \dots, \kappa$, $C=\{C_1, \dots, C_N\}$, with $\kappa=k-2$ greater than $2$.} in $B$, $S_\ell=\{x_{\ell,1}, \dots, x_{\ell,N}\}$, $\ell=1, \dots, k$, where $N=n(n+1)/2$ and $k$ is an integer larger than $4$. Then consider the initial covectors associated with all the geodesics joining the points in different sets.
The existence of a quadratic first-integral implies strong constraints on these covectors. If the points are in ``general'' position, small and localized perturbations of the metric along the geodesics make these constraints incompatible, which prevents the existence of a quadratic first-integral.

This argument is very general, it is not specific to Riemannian geometry. It only requires the following assumption on the $kN$ points:
\begin{description}
  \item[(H.1)] no three of the points $x_{1,1}, \dots, x_{k,N}$ lie on one normal geodesic;

  \item[(H.2)] for every sets $S_i \neq S_j$ and every point $x \in S_i$, there exists a 2-decisive set (see below) $p_1,\dots,p_N \in T_x^*B \simeq \R^n$ such that
      \[
      S_j=\{ \pi \circ e^{\vec{h}_g}(x,p_1),\dots, \pi \circ e^{\vec{h}_g}(x,p_N) \};
      \]

  \item[(H.3)] for every pair of sets $S_i \neq S_j$ and every pair of points $x \in S_i$ and $y \in S_j$, let $p \in T_{x}^*B \simeq \R^n$ be the covector such that $y = \pi \circ e^{\vec{h}_g}(x,p)$; then perturbations $\widetilde{g}$ of the metric $g$ localized near one point of the geodesic $\pi \circ e^{t\vec{h}_g}(x,p)$, $t\in (0,1)$, generate a neighborhood of $e^{\vec{h}_{g}}(x,p)$ in $T^*B$, i.e.\  the map
$$
\widetilde{g} \mapsto  e^{\vec{h}_{\widetilde{g}}}(x,p)
$$
is a submersion at $\widetilde{g}=g$.
\end{description}
As a consequence, if any sub-Riemannian metric $g$ admits $kN$ points satisfying \textbf{(H.1)}--\textbf{(H.3)}, then Proposition~\ref{le:KM} can be proved in the same way as \cite[Theorem 1]{Kruglikov-Matveev2015}. Thus we are reduced to proving the existence of such sets of points.

\begin{rk}
\label{rk:2decisif}
A set of $N=n(n+1)/2$ vectors of $\R^n$ is called \emph{2-decisive} if the values of any quadratic polynomial on this set determine the polynomial. Clearly, the set of 2-decisive sets is open and dense in the set of $N$-tuples of vectors of $\R^n$.
\end{rk}

Let us first study the perturbation property of \textbf{(H.3)}. We denote by $\mathcal{G}$ the set of sub-Riemannian $C^2$ metrics on $D$. Locally $\mathcal{G}$ can be identified with an open subset of the Banach space $\mathcal{S}$ of $C^2$ maps from $B$ to the set of symmetric $(m \times m)$ matrices.

\begin{Lemma}
\label{le:subm}
Let $g$ be a sub-Riemannian metric and $\lambda_0 \in T^*B$ be an ample covector with respect to $g$. Then the map
$$
\psi : \widetilde{g} \in \mathcal{G} \mapsto  e^{\vec{h}_{\widetilde{g}}}(\lambda_0) \in T^*B
$$
is a submersion at $\widetilde{g}=g$.
\end{Lemma}

\begin{proof}
From standard results on the dependance of differential equations with respect to a parameter, the differential of  $\psi$ at $g$ can be written as
\begin{equation*}
D_g \psi  : \widetilde{g} \in \mathcal{S} \mapsto {e_*^{\vec{h}_g}} \int_0^1 {e_*^{-s \vec{h}_g}} \left( \frac{\partial \vec{h}_g (\lambda(s))}{\partial g} (g) \cdot \widetilde{g}\right) ds,
\end{equation*}
where $\lambda(s) = e^{s \vec{h}_g} \lambda_0$, $s \in [0,1]$. Now, we can easily verify that, for a given $\lambda \in T^*B$, the image of the partial differential of $\vec{h}_g$ with respect to $g$ is
$$
\mathrm{Im} \left( \frac{\partial \vec{h}_g (\lambda)}{\partial g} (g)\right) = \left\{ v \in T_{\lambda}(T^*B) \ : \ \pi_* v \in D \right\} = J_\lambda^{(1)},
$$
where the last equality comes from Lemma~\ref{le:jacobi}. As a consequence, the image of the linear map $D_g \psi$ satisfies
$$
\mathrm{Im} D_g \psi = {e_*^{\vec{h}_g}} \int_0^1 {e_*^{-s \vec{h}_g}} J_{\lambda(s)}^{(1)} ds,
$$
and $\psi$ is a submersion at $g$ if and only if
\begin{equation}
\label{eq:spanJ1}
\mathrm{span} \left\{ {e_*^{-s \vec{h}_g}} J_{\lambda(s)}^{(1)} \ : \ s \in [0,1] \right\} = T_{\lambda_0}(T^*B).
\end{equation}

Assume by contradiction that \eqref{eq:spanJ1} does not hold. Then there exists $p \in T^*_{\lambda_0}(T^*B)$ such that $\langle p, {e_*^{-s \vec{h}_g}} J_{\lambda(s)}^{(1)} \rangle = 0$ for all $s \in [0,1]$. Note that $J_{\lambda_0}(s) \subset {e_*^{-s \vec{h}_g}} J_{\lambda(s)}^{(1)}$ (see Definition~\ref{def:jacobi}). Hence, for all smooth curve $l(\cdot)$ such that $l(s) \in J_{\lambda_0}(s)$ for all $s \in [0,1]$, we have $\langle p, l(s) \rangle \equiv 0$. Taking the derivatives with respect to $s$ at $0$, we get
$$
 \langle p, \frac{d^j}{dt^j}l(0) \rangle = 0 \quad \hbox{for all integer } j.
$$
From Definition~\ref{def:extension} this implies $\langle p, J_{\lambda_0}^{(k)} \rangle = 0$ for any integer $k$, which contradicts the fact that $\lambda_0$ is ample. Thus \eqref{eq:spanJ1} holds and $\psi$ is a submersion at $g$.
\end{proof}
As a direct consequence of this lemma, if \textbf{(H.2)} is satisfied with ample covectors $p_i$, then \textbf{(H.3)} is satisfied as well.

Let $x$ be a point in $B$ and $\exp_x$ be the exponential mapping at $x$, $\exp_x : p \in T^*_xB \to \pi \circ e^{\vec{h}_g} (x, p) \in B$. Since conjugate times are isolated from $0$ along a geodesic which is ample at $t=0$ (see for instance~\cite[Cor.\ 8.47]{ABB}), for any ample covector $p$ the map $\exp_x$ is locally open near $tp$ for $t$ small enough. Let us denote by $\mathcal{A}_x$ the set of $N$-tuples of ample covectors $(p_1,\dots,p_N)$ in $(T_x^*B)^N$ which are $2$-decisive, and set
$$
S(x) = \{ (\exp_x(p_1),\dots , \exp_x(p_N) ) \in B^N \ : \ (p_1,\dots,p_N) \in  \mathcal{A}_x \}.
$$
By Remark~\ref{rk:2decisif} and Theorem~\ref{th:Agr}, the set $\mathcal{A}_x$ is open and dense in $(T_x^*B)^N$. It results then from the local openness of the exponential map that $S(x)$ has a non empty interior with $(x, \dots,x)\in \overline{\mathrm{int}S(x)}$.
%

We are now in a position to give the construction of sets of $N$ points $S_\ell$, $\ell=1, \dots, k$, satisfying \textbf{(H.1)}--\textbf{(H.3)}. The properties above ensure that we can choose $S_1=\{x_{1,1}, \dots, x_{1,N}\} \in B^N$ such that no three points are aligned and such that the intersection
$$
\bigcap_{i=1}^N \mathrm{int} S(x_{1,i})
$$
is non empty. We then choose $S_2=\{x_{2,1}, \dots, x_{2,N}\}$ in this intersection such that no three points in $S_1 \cup S_2$ are aligned and such that the intersection of all sets $\mathrm{int}  S(x_{1,i}) \cap \mathrm{int} S(x_{2,i})$ is non empty. Iterating this construction we obtain $k$ sets of $N$ points satisfying \textbf{(H.1)}--\textbf{(H.3)}. This together with the argument in \cite{Kruglikov-Matveev2015} shows Proposition~\ref{le:KM} and then Proposition~\ref{le:quadratic}.

\end{document}